\documentclass[10pt,reqno]{amsart}
\usepackage{amssymb,mathrsfs,graphicx}
\usepackage{epsfig,subfigure}
\usepackage{ifthen}
\usepackage{colortbl}
\definecolor{black}{rgb}{0.0, 0.0, 0.0}
\definecolor{red}{rgb}{1.0, 0.5, 0.5}

\usepackage{ifthen} 
\provideboolean{shownotes} 
\setboolean{shownotes}{true} 
\newcommand{\margnote}[1]{
\ifthenelse{\boolean{shownotes}}%
{\marginpar{\raggedright\tiny\texttt{#1}}}%
{}%
}
\newcommand{\hole}[1]{
\ifthenelse{\boolean{shownotes}}%
{\begin{center} \fbox{ \rule {.25cm}{0cm} \rule[-.1cm]{0cm}{.4cm}
\parbox{.85\textwidth}{\begin{center} \texttt{#1}\end{center}} \rule
{.25cm}{0cm}}\end{center}} {} }

\title[Mean-field limit for collective behavior models with sharp sensitivity regions]{Mean-field limit for collective behavior models with sharp sensitivity regions}

\author[Carrillo]{Jos\'{e} A. Carrillo}
\address[Jos\'{e} A. Carrillo]{\newline Department of Mathematics
    \newline Imperial College London, London SW7 2AZ, United Kingdom}
\email{carrillo@imperial.ac.uk}

\author[Choi]{Young-Pil Choi}
\address[Young-Pil Choi]{\newline Fakult\"at f\"ur Mathematik
    \newline Technische Universit\"at M\"unchen, Boltzmannstra{\ss}e 3, 85748, Garching bei M\"unchen, Germany}
\email{ychoi@ma.tum.de}

\author[Hauray]{Maxime Hauray}
\address[Maxime Hauray]{\newline I2M, \newline
    Universit\'e d'Aix-Marseille, Centrale Marseille et CNRS UMR 7373, Marseille, France}
\email{maxime.hauray@univ-amu.fr}

\author[Salem]{Samir Salem}
\address[Samir Salem]{\newline I2M, \newline
    Universit\'e d'Aix-Marseille, Centrale Marseille et CNRS UMR 7373, Marseille, France}
\email{samir.salem@univ-amu.fr}

\numberwithin{equation}{section}

\newtheorem{theorem}{Theorem}[section]
\newtheorem{lemma}{Lemma}[section]
\newtheorem{corollary}{Corollary}[section]
\newtheorem{proposition}{Proposition}[section]
\newtheorem{remark}{Remark}[section]
\newtheorem{definition}{Definition}[section]

\newcommand{\R}{\mathbb R}
\newcommand{\N}{{\mathbb N}}

\newcommand{\bbr}{\mathbb R}

\newcommand{\pp}{\mathcal P}
\newcommand{\e}{\varepsilon}
\newcommand{\mt}{\mathcal{T}}

\newcommand{\mc}{\mathcal{C}}
\newcommand{\ep}{\varepsilon}
\newcommand{\lt}{\left}
\newcommand{\rt}{\right}
\newcommand{\pa}{\partial}
 
\newcommand{\mb}{\mathbf{1}}

\newcommand{\bq}{\begin{equation}}
\newcommand{\eq}{\end{equation}}
\newcommand{\om}{\Omega}
\newcommand{\mmz}{\mathcal{Z}}
\newcommand{\mmb}{\mathcal{B}}

\def\charf {\mbox{{\text 1}\kern-.30em {\text l}}}

\def\wpa{\widetilde \partial}

\begin{document}
\allowdisplaybreaks

\date{\today}



\begin{abstract} 
We rigorously show the mean-field limit for a large class of swarming individual based models with local sharp sensitivity regions. For instance, these models include nonlocal repulsive-attractive forces locally averaged over sharp vision cones and Cucker-Smale interactions with discontinuous communication weights. We construct global-in-time defined notion of solutions through a differential inclusion system corresponding to the particle descriptions. We estimate the error between the solutions to the differential inclusion system and weak solutions to the expected limiting kinetic equation by employing tools from optimal transport theory. Quantitative bounds on the expansion of the 1-Wasserstein distance along flows based on a weak-strong stability estimate are obtained. We also provide different examples of realistic sensitivity sets satisfying the assumptions of our main results. 
\end{abstract}

\maketitle \centerline{\date}


%
%
%
%
\section{Introduction}\label{intro}

In this work, we deal with the rigorous derivation of kinetic models of collective motion when sharp sensitivity regions are considered. Most of the minimal particle-like models for swarming include the effects of attraction, repulsion, and alignment by means of local averages of these kind of forces around the location or velocity of the individuals. These Individual Based Models (IBMs) with local sensitivity regions have been proposed in the theoretical biology literature \cite{HH,HCH} as well as in the applied mathematical community \cite{LLE,AIR}.
For instance, a basic model introducing the three effects combines the repulsive-attractive forces modeled by an effective pairwise interaction potential $\varphi(x)$ and the gregariousness behavior of individuals by locally averaging their relative velocities with weights depending on the interindividual distances, see \cite{CFTV,ABCV} for instance. These effects lead to pattern formations such as localized flocks, which are stable for the particle dynamics as proven in \cite{ABCV,CHM}. This model for alignment is usually referred as the Cucker-Smale model for velocity consensus \cite{CS} that has been throughly studied in the last years, see for instance \cite{CFTV}. The Cucker-Smale model averages the relative velocities through a coupling function $h(v)$, linear in the original model, and a communication weight function $\psi(x)$ weighting the relative importance in the orientation and speed averaging according to the interparticle distances. In summary, given $N$ particles the evolution of their positions $X_i(t)$ and velocities $V_i(t)$ is determined by the 2nd order ODE system:
\begin{equation}\label{odesintro}
\left\{ \begin{array}{ll}
\dot{X}_i(t) = V_i(t), \quad i=1,\cdots,N,~~t > 0, &  \\[2mm]
\displaystyle\dot{V}_i(t) = \frac1N \sum_{j \neq i} \left[\psi(X_i-X_j) h(V_j - V_i) 
+ \nabla_x \varphi(X_i-X_j)\right] \mb_{K(V_i)}(X_i-X_j) 
&\\
\left( X_i(0), V_i(0) \right) =: \left( X_i^0, V_i^0 \right),\quad i=1,\cdots,N. & 
\end{array} \right.
\end{equation}
Here, $\mb_{K(v)}$ is the indicator function on the set $K(v)\subset \R^d$. The local averaging for attraction, repulsion and reorientation is done on velocity dependent sets $K(v)$ that we refer as the sensitivity regions, a typical case being vision cones for animals such as birds or mammals. However, sensitivity regions can highly vary across animals due to their different sensory organs, for instance pressure waves through lateral lines for fish or visual field for birds could depend a lot on the particular specie, see for instance \cite{KTIHC,LLE2}.

In this work, we are interested in the continuum mean-field models obtained from these IBMs in the large particle limit, i.e. when $N\to\infty$. The mean-field limit in statistical mechanics is obtained when the total force felt by each individual is of unit order when $N\to\infty$, this means that the forces exerted by each individual on others have to be scaled inversely proportional to the number of individuals, see \cite{Spohn} for further explanations about the mean-field scaling. The mean-field limit leads as usual to kinetic equations of Vlasov-type. Their formal derivation has been done for IBMs of collective behavior introducing these kind of sensitivity regions by different authors \cite{HT,CFTV,AIR,AP,MT,MT2,Has} using BBGKY hierarchies arguments. These  models include locally averaged interactions in the metric sense \cite{HT,CFTV,AIR,MT,MT2} and in the topological sense \cite{AP, Has}. 

On the other hand, trying to prove rigorously that this mean-field limit holds or equivalently trying to show rigorously the validity of the kinetic equations obtained from the formal BBGKY hierarchy procedure is a very challenging question. The lack of enough regularity of the velocity field, due to the sharpness of the sensitivity regions, prevents classical arguments for deriving rigorously the mean-field limit such as \cite{dobru,BH,Neun,Spohn,Szni,Glose} to be applicable. The different variants of these strategies introduced in \cite{HL,BCC,CCR1,CCR2} deal with locally Lipschitz velocity fields to allow swarming systems such as the Cucker-Smale alignment models \cite{CS,PKH,ACHL}. However, these arguments do not apply either for discontinuous sharp interaction regions between individuals. There have been several important works dealing with singular interaction forces either at the first order or the second order models \cite{Hauray,HJ1,HJ2,CCH,CCH2} but none dealing with discontinuous averaging regions.

Our work is the first in showing the rigorous proof of the mean-field limit of a system of interacting particles where each particle only interacts with those inside a local region whose shape depends on the position and velocity of the particle. Even the meaning of the particle model is not clear due to the discontinuous right-hand side of the dynamical system. We can show that Filippov's theory \cite{Filippov} allows us to have a well-defined notion of solutions via differential inclusions. We will exemplify our main results in the Cucker-Smale model \cite{CS} with local sensitivity regions, given by 
\begin{equation}\label{sys_kin}
\pa_t f + v \cdot \nabla_x f + \nabla_v \cdot (F(f)f) = 0, \quad (x,v) \in \R^d \times \R^d,\,\, t > 0,
\end{equation}
with the initial data:
\begin{equation}\label{ini_sys_kin}
f(x,v,t)|_{t=0} = f_0(x,v).
\end{equation}
Here, $F(f)$ is an alignment force trying to average the relative velocities as
\[
F(f)(x,v,t) = \int_{\R^d \times \R^d} \mb_{K(v)}(x-y) (w - v) f(y,w)\,dydw\,.
\]
Under suitable assumptions of the local sensitivity regions $K(v)$, we will show that the velocity field associated to an $L^1-L^\infty$ density $f$ is locally Lipschitz continuous. This is the basic property that allows for a weak-strong quantitative stability estimate between weak solutions of the kinetic continuum model \eqref{sys_kin} and the particle solutions of the discrete differential inclusion system. This estimate implies readily the mean-field limit derivation of the kinetic equation \eqref{sys_kin}. These results are easily generalizable for the general kinetic model associated to \eqref{odesintro}: 
\begin{equation*}
\left\{ \begin{array}{ll}
\pa_t f + v \cdot \nabla_x f + \nabla_v \cdot (F(f)f) = 0, \qquad (x,v)\in \R^d \times \R^d, \quad t > 0,& \\[2mm]
f(x,v,t)|_{t=0}=f_0(x,v)\qquad (x,v) \in \R^d \times \R^d,&  
\end{array} \right.
\end{equation*}
with
\begin{equation*}
F(f)(x,v,t) = \int_{\R^d \times \R^d} \left[\psi(x-y) h(w - v) + \nabla_x \varphi(x-y)\right] \mb_{K(v)}(x-y) f(y,w)\,dydw .
\end{equation*}

This paper is organized as follows: we first show the existence of the differential inclusion system and the crucial  Lipschitz property of $F(f)$ for $L^1-L^\infty$ densities $f$ under very reasonable but technical assumptions on the sensitivity regions. We explain carefully the necessary assumptions on the sensitivity sets in Section 2 before we state the main results of this work. Section 3 is devoted to show the weak-strong stability estimate and the mean-field limit. Section 4 shows the well-posedness of weak solutions for system \eqref{sys_kin}-\eqref{ini_sys_kin} under the same assumptions on the sensitivity regions. Section 5 gives some typical examples of local sensitivity regions verifying our assumptions commonly used in applications in mathematical biology. Finally, we extend these results to general locally averaged potential and alignment models on sharp sensitivity regions in Section 6.


\section{Preliminaries \& Main Results}

In order to deal with the differential inclusion system associated to the IBMs discussed in the introduction and their associated kinetic equations, we first need some notation to treat rigorously the dependence on the sensitivity sets.

\subsection{Sensitivity Regions: Notation}
\begin{definition}Let $K \subset \R^d$ be a non-empty compact set and $\e >0$. We define the $\e$-boundary of $K$ by:
\[
\partial^{\e}K:= \pa K + \overline{B(0,\e)}=\left\{x+y \ | \ x\in \partial K, |y|\leq \e \right\},
\] 
and also the $\e$-enlargement(resp. $\e$-reduction) $K^{\e,+}$ (resp. $K^{\e,-}$) by
\[
K^{\e,+}:=\partial^{\e}K\cup K = K +  \overline{B(0,\e)} \quad \mbox{and} \quad K^{\e,-}:=K\setminus \partial^{\e} K.
\]
\end{definition}
Note that $K^{\e,-}$ can be the empty set and that $\pa^\e K = K^{\e,+} \setminus K^{\e,-}$. We sketch these definitions in Figure \ref{fig1}. We denote the Lebesgue measure of set $A$ in $\R^d$ by $|A|$.

\begin{figure}[ht]
        \centering
        \mbox{
         {\includegraphics[width=8cm,height=6cm]{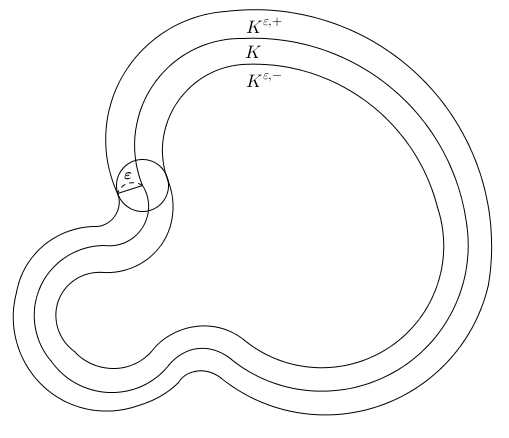}}}
         \caption{A sketch of the sets $K^{\e,-}, K$, and $K^{\e,+}$.}
         \label{fig1}
\end{figure}

We are going to show several elementary estimates, some of them obvious by geometrical arguments.
\begin{lemma}\label{lem_est1}For $\e > 0$ and $\delta > 0$, it holds
\[
(\pa^\e K)^{\delta,+} \subset \pa^{\e + \delta}K.
\]
\end{lemma}
\begin{proof} For $x \in (\pa^\e K)^{\delta,+}$, there exists a $y \in \pa^\e K$ such that $d(x,y) \leq \delta$ by definition. This yields
\[
d(x,K) \leq d(x,y) + d(y,K) \leq \delta + \e,
\]
i.e., $x \in K^{(\e + \delta),+}$. We now suppose $x \in K^{(\e + \delta),-}$, i.e., $d(x,K^c) > \e + \delta$. Then for all $y \in \pa^\e K$ we obtain
\[
\e + \delta < d(x,K^c) \leq d(x,y) + d(y,K^c) \leq d(x,y) + \e.
\]
This is a contradiction to $d(x,y) \leq \delta$. Hence we have $x \notin K^{(\e + \delta),-}$, and this completes the proof.
\end{proof}

\begin{lemma}\label{lem_est2}For $K\subset \mathbb{R^d}$ and  $x_1,y_1,x_2,y_2 \in \R^d$, define $\e_1 = |x_1 - x_2|$, $\e_2 = |y_1 - y_2|$, we get
\[
|\mb_K(y_1 - x_1) - \mb_K(y_2 - x_2)| \leq \mb_{\pa^{2\max(\e_1,\e_2)}K}(y_1 - x_1) \leq \mb_{\pa^{2 \e_1}K}(y_1 - x_1) + \mb_{\pa^{2 \e_2}K}(y_1 - x_1).
\]
\end{lemma}
\begin{proof}Since
$\pa^{2\max(\e_1,\e_2)}K = \pa^{2\e_1} K \cup \pa^{2\e_2} K$, we deduce
\[
\mb_{\pa^{2\max(\e_1,\e_2)}K}(y_1 - x_1) \leq \mb_{\pa^{2\e_1}K}(y_1 - x_1) + \mb_{\pa^{2\e_2}K}(y_1 - x_1).
\]
Thus, it is enough to show
\[
|\mb_K(y_1 - x_1) - \mb_K(y_2 - x_2)| \leq \mb_{\pa^{2\max(\e_1,\e_2)}K}(y_1 - x_1).
\]
We separate cases:

$\bullet$ If $y_1 - x_1 \notin K^{2\max(\e_1,\e_2), +}$, then $d(y_1 - x_1,K) >  2\max(\e_1,\e_2)$ and
\[
|y_1 - x_1 - (y_2 - x_2)| \leq \e_1 + \e_2 \leq 2\max(\e_1,\e_2)\,,
\]
from which we deduce that
\[
y_2 - x_2 \notin K \quad \mbox{and} \quad \mb_K(y_1 - x_1) = \mb_K(y_2 - x_2) = 0.
\]

$\bullet$ If $y_1 - x_1 \in K^{2\max(\e_1,\e_2),-} \subset K$, we get
\[
|y_1 - x_1 - (y_2 - x_2)| \leq 2\max(\e_1,\e_2) < d(y-x,K^c),
\]
and this implies $y_2 - x_2 \in K$. Thus we obtain
\[
\mb_K(y_1 - x_1) - \mb_K(y_2 - x_2) = 0.
\]

$\bullet$ If $y_1 - x_1 \in K^{2\max(\e_1,\e_2), +} \setminus K^{2\max(\e_1,\e_2),-}$, it is trivial to get
\[
|\mb_K(y_1 - x_1) - \mb_K(y_2 - x_2)| \leq \mb_{\pa^{2\max(\e_1,\e_2)}K}(y_1 - x_1).
\]
This concludes the desired result.
\end{proof}

\subsection{Sensitivity Regions: Assumptions}

Let us denote by $A \Delta B$ the symmetric difference between two sets $A \subset \R^d$ and $B \subset \R^d$, i.e., 
\[
A \Delta B = (A\setminus B) \cup (B \setminus A) = (A \cup B)\setminus (A \cap B).
\]
It is obvious that $|\mb_{A}(x) - \mb_{B}(x)| = \mb_{A \Delta B}(x)$.

In this work, we will assume that the sensitivity region set-valued function $K(v)$ depending on the velocity variable satisfies the following conditions: 
\begin{itemize}
\item ${\bf (H1)}$: $K(\cdot)$ is globally compact, i.e., $K(v)$ is compact and there exists a compact set $\mathcal{K}$ such that $K(v)\subseteq \mathcal{K}\,, \forall v\in \mathbb{R}^d$.

\item ${\bf (H2)}$: There exists a family of  closed sets $v \mapsto \Theta(v)$  and a constant $C$ such that:
\begin{itemize}
\item $(i)$ $\partial K(v) \subset \Theta(v) $, for all $v \in \R^d$,
\item $(ii)$ $|\Theta(v)^{\e,+}| \le C \e$, for all $\e \in (0,1)$,
\item $(iii)$ $K(v) \Delta K(w) \subset \Theta(v)^{C |v-w|,+}$, 
for $v,w \in \R^d$,
\item $(iv)$ $\Theta(w) \subset \Theta(v)^{C |v-w|,+}$, 
for $v,w \in \R^d$.
\end{itemize}
\end{itemize}

\begin{remark}
Note that ${\bf (H2)}$-$(ii)$ is satisfied if $\Theta(v)$ is made for piecewise
$\mathcal{C}^1$-compact hypersurfaces. That will be the case in all our applications. In general it roughly means that $\Theta(v)$ has finite $(d-1)$-area, i.e. it has a $(d-1)$-dimensional Hausdorff measure bounded by some constant. 
\end{remark}

We will discuss several examples of sets satisfying these hypotheses in Section \ref{sec:5}: a fixed ball, a ball with a radius depending on the speed, a vision cone in which the angle depends on the speed. Let us just mention here that the hypotheses are fulfilled in the constant case $K(v)=K_0$ when $\partial K_0$ is made of compact $\mathcal{C}^1$ hypersurfaces. In fact, in this case it is enough to choose $\Theta(v) = \partial K_0$, and all the hypothesis are satisfied (many of them are empty in this particular case), see Section \ref{sec:5_1}.

\begin{remark}One of most common choices for sensitivity region in applications is a vision cone. Let us define
\[
C(r,v,\theta) = \lt\{ x : |x| \leq r, \quad \frac{x \cdot v}{|x||v|} \geq \cos \theta\rt\},
\]
with $\theta \in (0, \pi)$, $r > 0$. Note that a cone  with fixed radius and angle is a $0$-homogeneous set-valued function in $v$, that is,
\[
C(r,v,\theta) = C\lt(r,\frac{v}{|v|},\theta\rt) \quad \forall \,v \in \R^d \setminus \{0\}.
\] 
Then, it is not difficult to see that $v \mapsto C(r,v,\theta)$ is discontinuous at the origin, and ${\bf (H2)}$-$(ii),(iii)$ which implies that $| K(v) \Delta K(w) | \le C|v-w|$ cannot be satisfied except in the case where the cone is a ball. This discontinuity due to the vanishing velocity is not important from the modelling view point since most of swarming models have additional terms preventing individuals to stop \cite{CFTV, AIR, AP}. We will discuss in Section {\rm\ref{sec:5}} very simple and mild variations of vision cone regions that do satisfy ${\bf (H2)}$-$(ii)$ and that are innocuous from the modelling viewpoint in swarming.
\end{remark}

\begin{remark} The additional set valued function $\Theta(v)$ is an enlargement of the boundary $\partial K(v)$ for any $v \in \R^d$, and it is requested because we cannot work only with the boundary set $\partial K(v)$ for two reasons. In fact, there is first a natural set, larger than $\partial K(v)$, that appears in many calculations. Let us define it as
\begin{equation}\label{def:wpaK}
\widetilde{\partial}K(v) :=  \lt\{ x \in \R^d, (v,x) \in \partial \widetilde K \rt\} 
, \quad \text{where} \quad \widetilde K := \lt\{ (v,x), \; x \in K(v) \rt\}.
\end{equation}
Thus it is a slice of the boundary of the set $\widetilde K$. Note that $\pa K(v) \subset \widetilde{\partial}K(v)$.

The set $\widetilde{\partial}K(v)$ is in fact important in order to define the particle system associated to discontinuous kernels. More precisely, we need to give a sense to the derivative in time of the particle paths when they cross the boundary of the set $K(v)$. In order to do that, we will work with differential inclusions, and thus we need to take into account the set of admissible slopes at the boundary of the set $K(v)$. We define the set-valued function of admissible slopes as
\[
\mathcal{A}(x,v) = Conv\lt\{ \alpha \in [0,1] : \exists\, (x^n,v^n) \to (x,v) \mbox{ such that } \mb_{K(v^n)}(x^n) \to \alpha\rt\},
\]
for $(x,v) \in \R^d \times \R^d$, where $Conv$ denotes the convex hull of a set. We remark that there are only three different values for $\mathcal{A}(x,v)$: two singleton sets $\{0\}$, $\{1\}$, and the unit interval $[0,1]$. Indeed, we have
\[
\widetilde{\partial}K(v) = \lt\{ x \in \R^d : \mathcal{A}(x,v) = [0,1] \rt\}.
\]

Note that $\widetilde{\partial}K(v)$ can be strictly larger than $\pa K(v)$, see for instance the example of the vision cone in Section~{\rm\ref{sec:5}}. For example, if we take into account the cone $K(v) = C(r,v, \alpha(|v|))$ with fixed $r$ but with varying angles $\alpha(|v|)$ such that $\alpha(|v|) \to \pi$ as $|v| \searrow 1$, then the generalized boundary set $\widetilde{\partial}K(v)$ is strictly larger than $\pa K(v)$ (see Figure {\rm \ref{fig2}}).

\begin{figure}[lt]
        \centering
        \mbox{
         {\includegraphics[width=4.65cm,height=4.55cm]{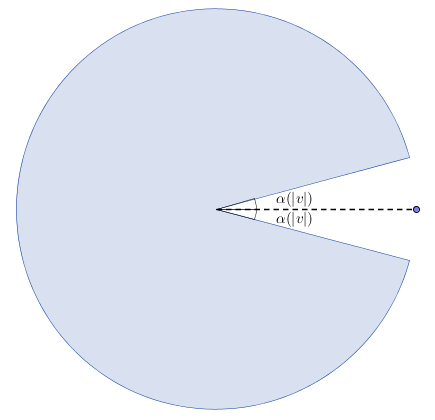}} \quad 
         {\includegraphics[width=4.5cm,height=4.5cm]{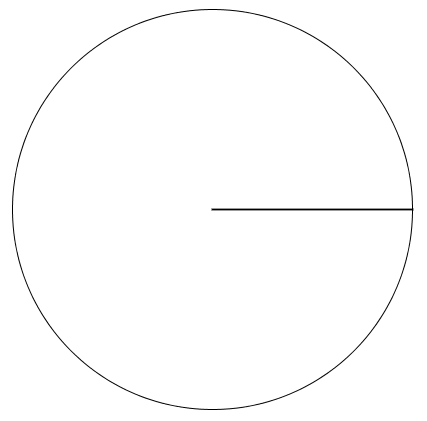}}\quad
         {\includegraphics[width=4.6cm,height=4.6cm]{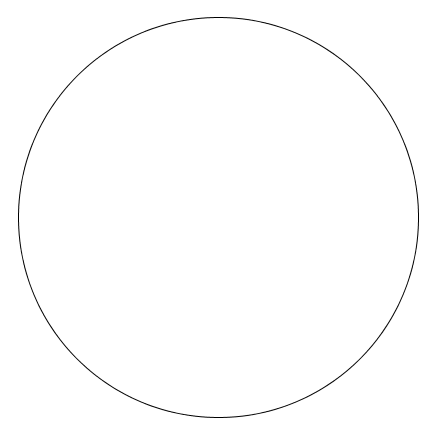}}
         }
         \caption{Evolution of $\widetilde{\partial}K(v)$ when $K(v) = C(r,v,\theta(|v|)$ for (from left to right) $|v|>1$, $|v|=1$ and $|v|<1$. $\wpa K(v)$ is strictly larger than $\partial K(v) $ when $|v|=1$.} 
         \label{fig2}
\end{figure}

On the other hand, we may also need to enlarge a little bit $\wpa K(v)$, since it may not satisfy the hypothesis $({\bf H2})$-$(iv)$, that is more or less a kind of Lipschitz property for the set valued function $\Theta$.
This is why we need to introduce the additional set $\Theta(v)$ which does contain $\wpa K(v)$, as we prove next, even if it is not explicitly written in the assumption $({\bf H2})$. 
\end{remark}

\begin{lemma}\label{rem:BsubThe}
Under the assumptions $(\bf{H2})$-$(i), (ii),$ and $(iii)$, the set valued function $\wpa K$ defined in~\eqref{def:wpaK} satisfies
\[
\wpa K(v)\subseteq \Theta(v) \ , \forall v\in\R^d.
\]
\end{lemma}
\begin{proof}
Let $v\in\R^d$ and $x\in \widetilde{\partial}K(v)$. By definition of $\widetilde{\partial}K(v)$, there exists two sequences $(x^1_n,v^1_n)_{n \in \N}$ and $(x^2_n,v^2_n)_{n \in \N}$ both converging to $(x,v)$ such that $x_n^1\in K(v_n^1)$ and $x_n^2\notin K(v_n^2)$.

If $x \in \partial K(v)$ then it is also in $\Theta(v)$ by the assumption $(\textbf{H2})$-$(i)$. 
If $x\in \overset{\circ}{K(v)}$, then for $n$ large enough $x_n^2\in \overset{\circ}{K(v)}$. Therefore $x_n^2\in K(v_n^2)\Delta K(v)\subset \Theta(v)^{C|v_n^2-v|,+}$ due to $(\textbf{H2})$-$(iii)$. Thus $x\in \Theta(v)^{C|v_n^2-v|+|x_n^2-x|,+}$, and letting $n$ go to infinity we deduce that $x\in \Theta(v)$ since $\Theta(v)$ is closed. 
If $x\in K(v)^c$, we consider the $(x_n^1,v_n^1)_{n \in \N}$ and repeat the previous step.  
\end{proof} 

\subsection{Particle system: Differential inclusions}\label{sec_23}

A particle approximation of the kinetic equation \eqref{sys_kin} should read as the following system:
\begin{equation}\label{sys_ode}
\left\{ \begin{array}{ll}
\dot{X}_i(t) = V_i(t), & \\[2mm]
\displaystyle \dot{V}_i(t) = -\sum_{j \neq i}m_j \mb_{K_{(V_i)}}(X_i-X_j)(V_i - V_j), & i=1,\cdots,N,~~t > 0,\\[5mm]
\left( X_i(0), V_i(0) \right) =: \left( X_i^0, V_i^0 \right) & i=1,\cdots,N.
\end{array} \right.
\end{equation}
Here $\{ (X_i,V_i) \}_{i=1}^N$ and $\{ m_i\}_{i=1}^N$ are the position, velocity, and weight of $i$-th particles, respectively. We will always assume the normalization condition $\sum_{i=1}^{N}m_i = 1$, with $m_i > 0$, $i=1,\cdots,N$. However, this particle system is not well-defined due to the discontinuous character of the right-hand side. For the mean-field limit, 
we will use the global solutions to the following differential inclusion system generalizing the ODE system \eqref{sys_ode}:
\bq\label{sys_di}
\dot{\mmz}^N (t) \in \mmb^N(\mmz^N (t)),
\eq
where $\mmb^N : \R^{2dN} \to P(\R^{2dN})$ is a set-valued given by
\[
\mmb^N(x_1, \cdots, x_N, v_1, \cdots, v_N) := \{v_1\} \times \cdots \times \{v_N\} \times \mathcal{F}_1(x,v) \times \cdots \times \mathcal{F}_N(x,v),
\]
with
$$
\mathcal{F}_i(x,v):=\bigg\{ \sum_{j \neq i} \alpha_{ij}m_j(v_j - v_i) : \alpha_{ij}\in \mathcal{A}(x_i - x_j,v_i) \bigg\},
$$
for $i \in \{1,\cdots,N\}$. A function $\mmz^N(t)$ is called a solution of the differential inclusion system \eqref{sys_di} with initial condition $\mmz^N(0) = \mmz_0^N$ if 
\[
\mmz^N(t) = \mmz_0^N + \int_0^t \mathcal{F}^N(s)\,ds, \quad t \geq 0, \quad \mbox{and} \quad \mathcal{F}^N(t) \in \mmb^N(\mmz^N(t)) \quad \mbox{a.e.} \quad t \geq 0,
\]
which is equivalent to 
\bq\label{eq_di}
\dot{\mmz}^N(t) = \mathcal{F}^N(t) \quad \mbox{and} \quad \mathcal{F}^N(t) \in \mmb^N(\mmz^N(t)) \quad \mbox{a.e.} \quad t \geq 0.
\eq

Using the following Filippov theory of discontinuous dynamical systems \cite{Filippov}, we will show that the differential inclusion system \eqref{sys_di} has global-in-time solutions.

\begin{theorem}\label{thm_fil}\cite{Filippov} Let $F:(x,t)\in \R^n \times \R_+ \mapsto F(x,t)\in P(\R^n)$ be a set-valued function such that 
\begin{itemize}
\item $F(x,t)$ is nonempty convex and closed for all $(x,t)\in \R^n \times \R_+$.
\item There exists a bounded set $K \subset \R^n$ such that $F(x,t)\subseteq K$ for all $(x,t)\in\R^n\times\R_+$.
\item $F$ is upper continuous with respect to the inclusion.
\end{itemize}
Then, for all $(x_0,t_0)\in \R^n \times \R_+$, there exists at least one solution to the problem:
	\[
	\dot{x}(t)\in F(x(t),t) \quad  \mbox{with} \quad x(t_0)=x_0.
	\]
\end{theorem}
Thanks to this result, we provide global existence of solutions in time to the differential inclusion system \eqref{sys_di} in the proposition below.
\begin{proposition}\label{prop_dis} For any initial data $\mmz^N(0)$, there exists at least one global solution to the differential inclusion system \eqref{sys_di}.
\end{proposition}
\begin{proof}
It follows from the definition of $\mmb^N$ that $\mmb^N$ is bounded, closed, and convex for each $\mmz^N \in \R^{2dN}$.
We next easily get the upper continuity of $\mmb^N$ with respect to the inclusion from our construction of the differential inclusion system. For $R > 0$, we consider the truncated field $\mmb^N_R$ defined as
\[
\mmb^N_R:= \{v^R_1\} \times \cdots \times \{v^R_N\} \times \mathcal{F}^R_1(x,v) \times \cdots \times \mathcal{F}^R_N(x,v),
\]	
where 
\[
v_i^R:=\left\{ \begin{array}{ll}
\displaystyle \frac{R\wedge |v_i|}{|v_i|}v_i & \textrm{if $\,\, v_i \neq 0$}\\[2mm]
0 & \textrm{otherwise}
  \end{array} \right.
  \quad \mbox{and} \quad 
\mathcal{F}^R_i(x,v):=\bigg\{ \sum_{j \neq i} \alpha_{ij}m_j(v^R_j - v^R_i) : \alpha_{ij}\in \mathcal{A}(x_i - x_j,v_i) \bigg\}.
\]
Here $a \wedge b := \min\{a,b\}$ for $a,b \in \R_+$. Then, by Theorem \ref{thm_fil}, we have global solutions to the following differential inclusion system:
\bq\label{sys_di2}
\dot{\mmz}^N(t) \in \mmb^N_R(\mmz^N(t)) \quad \mbox{with the initial data} \quad \mmz^N_0.
\eq
Then we use the solution to the system \eqref{sys_di2} to show that the fields $\mmb^N$ and $\mmb^N_R$ generate the same solutions for $R>0$ large enough, i.e., $\mathcal{Z}^N(t)$ is also the solution to \eqref{sys_di}. Note that if we choose an index $i$ such that $|V_i^N(t)| = \max_{1 \leq j \leq N}|V_{j}^N(t)|$ for each $t$, then it follows from \eqref{sys_di} that for some $a_{ij}(t) \in \mathcal{A}(X_i^N(t) - X_j^N(t),V_i^N(t))$
\[
\frac{d}{dt}|V_i^N(t)|^2 = \sum_{j \neq i}a_{ij}(t)m_j(V_j^N(t) - V_i^N(t))\cdot V_i^N(t) \leq 0,
\]
due to the positivity of $a_{ij}$ and the choice of the index $i$. This yields
\[
\max_{1 \leq i \leq N}|V_i^N(t)| \leq \max_{1 \leq i \leq N}|V_i^N(0)| \quad \mbox{for a.e. } t\geq 0,
\]
Finally, by choosing $R= \max_{1\leq i \leq N}|V_i(0)|$, we conclude the existence of solutions to the system \eqref{sys_di}. 
\end{proof}

\subsection{MKR-distance} 

In order to establish quantitative estimates between solutions of the kinetic equation \eqref{sys_kin} and the differential inclusion \eqref{sys_di}, we need some basic tools of optimal transportation. 

\begin{definition}(Monge-Kantorovich-Rubinstein distance) \label{defdp}
Let $\rho_1,~ \rho_2$ be two Borel probability measures on $\bbr^d$. 
Then the Monge-Kantorovich-Rubinstein distance distance between $\rho_1$ and $\rho_2$ is defined as
\begin{equation*}\label{d1}
d_1(\rho_1,\rho_2) := \inf_{\gamma} \int_{\bbr^d \times
\bbr^d} |x-y| \, d\gamma(x,y) ,
\end{equation*}
where the infimum runs over all transference plans, i.e., all
probability measures $\gamma$ on $\bbr^d \times \bbr^d$ with
marginals $\rho_1$ and $\rho_2$ respectively,
\[
\int_{\bbr^d \times \bbr^d} \phi(x) d\gamma(x,y) = \int_{\bbr^d}
\phi(x) \,d\rho_1(x),
\]
and
\[
\int_{\bbr^d \times \bbr^d} \phi(y) d\gamma(x,y) = \int_{\bbr^d}
\phi(y) \,d\rho_2(y),
\]
for all $\phi \in \mathcal{C}_b(\bbr^d)$.
\end{definition}
Note that $\mathcal{P}_1(\R^d)$, the set of probability measures in $\bbr^d$ with first bounded moment, is a complete metric space endowed with the Monge-Kantorovich-Rubinstein distance. The Monge-Kantorovich-Rubinstein distance, also called 1-Wasserstein distance, is also equivalent to the Bounded Lipschitz distance 
\begin{equation}\label{blipd}
d_1(\rho_1,\rho_2) = \sup\left\{ \lt|\int_{\R^d} \varphi(\xi)d\rho_1(\xi) -  \int_{\R^d} \varphi(\xi)d\rho_2(\xi)\rt| \,\Big|\, \varphi \in \mbox{Lip}(\R^d), \mbox{Lip}(\varphi) \leq 1\right\},
\end{equation}
where Lip($\R^d$) and Lip($\varphi$) denote the set of Lipschitz functions on $\R^d$ and the Lipschitz constant of a function $\varphi$, respectively. We also remind the definition of the push-forward of a measure by a mapping in order to give the relation between Wasserstein distances and optimal transportation.

\begin{definition}
Let $\rho_1$ be a Borel measure on $\bbr^d$ and $\mathcal{T} :
\bbr^d \to \bbr^d$ be a measurable mapping. Then the push-forward
of $\rho_1$ by $\mathcal{T}$ is the measure $\rho_2$ defined by
\[
\rho_2(B) = \rho_1(\mathcal{T}^{-1}(B)) \quad \mbox{for} \quad B
\subset \bbr^d,
\]
and denoted as $\rho_2 = \mathcal{T} \# \rho_1$.
\end{definition}

We recall in the next proposition some classical properties, which proofs may be found in \cite{Vil}.

\begin{proposition}\label{prop-proper}
(i) The definition of $\rho_2 = \mathcal{T} \# \rho_1$ is equivalent
to
$$
\int_{\bbr^d} \phi(x)\, d\rho_2(x) = \int_{\bbr^d} \phi(\mathcal
T(x))\, d\rho_1(x)
$$
for all $\phi\in \mathcal{C}_b(\R^d)$. Given a probability measure with
first bounded moment $\rho_0$, consider two measurable mappings
$X_1,X_2 : \bbr^d \to \bbr^d$, then the following inequality
holds:
\[
d_1(X_1 \# \rho_0, X_2 \# \rho_0) \leq \int_{\bbr^d \times
\bbr^d} |x-y| d\gamma(x,y) = \int_{\bbr^d} | X_1(x) - X_2(x)|
d\rho_0(x).
\]
Here, we used as transference plan $\gamma = (X_1 \times
X_2)\#\rho_0$ in Definition \ref{defdp}. \newline

\noindent (ii) Given $\{\rho_k\}_{k=1}^{\infty}$ and $\rho$ in $\pp_1(\R^d)$, the following statements are equivalent:
\begin{itemize}
\item $d_1(\rho_k,\rho) \to 0$ as $k \to +\infty$.
\item $\rho_k$ converges to $\rho$ weakly-* as measures and
\[
\int_{\R^d}|\xi|d\rho_k(\xi) \to \int_{\R^d} |\xi| d\rho(\xi), \quad \mbox{as} \quad k \to +\infty.
\]
\end{itemize}
\end{proposition}

\subsection{Limiting Kinetic Equation: Well-posedness}\label{sec_2.5}
Let us define the notion of weak solution to the expected limiting kinetic equation \eqref{sys_kin}.

\begin{definition}\label{def_weak}For a given $T \in (0,\infty)$, f is a weak solution of the equation \eqref{sys_kin} on the time-interval $[0,T)$ with the sensitivity
region set-valued function $K(v)$ satisfying ${\bf (H1)}$-${\bf (H2)}$ if and only if the following conditions are satisfied:
\begin{enumerate}
\item $f \in L^\infty(0,T;(L_+^1 \cap L^\infty)(\R^d \times \R^d)) \cap L^\infty(0,T;\mathcal{P}_1(\R^d \times \R^d))$.
\item For all $\Psi \in \mc^\infty_c(\R^d \times \R^d \times [0,T])$, 
$$\begin{aligned}
&\int_{\R^d \times \R^d} f(x,v,T)\Psi(x,v,T)\,dxdv - \int_0^T \int_{\R^d \times \R^d} f(\pa_t \Psi + \nabla_x \Psi \cdot v + \nabla_v \Psi \cdot F(f))\,dxdvdt\cr
&\qquad = \int_{\R^d \times \R^d} f_0(x,v)\Psi_0(x,v)\,dxdv,
\end{aligned}$$
where $\Psi_0(x,v) := \Psi(x,v,0)$.
\end{enumerate}
\end{definition}

The main theorem regarding the well-posedness of the limiting kinetic equation is based on a weak-strong stability estimate in $d_1$. Similar estimates were recently obtained in \cite{Hauray} in the case of the Vlasov-Poisson system in one dimension.

\begin{theorem}\label{thm_weak}
Given an initial data satisfying 
\bq\label{main_ass}
f_0 \in (L^1_+ \cap L^\infty)(\R^d \times \R^d) \cap \mathcal{P}_1(\R^d \times \R^d)\,, 
\eq
and assume further that $f_0$ is compactly supported in velocity. Then there exists a positive time $T > 0$ such that the system \eqref{sys_kin}-\eqref{ini_sys_kin} with the sensitivity region set-valued function $K(v)$ satisfying ${\bf (H1)}$-${\bf (H2)}$ admits a unique weak solution $f$ in the sense of Definition {\rm \ref{def_weak}} on the time interval $[0,T]$, which is also compactly supported in velocity. Moreover, $f$ is determined as the push-forward of the initial density through the flow map generated by the local Lipschitz velocity field $(v, F(f))$ in phase space. Furthermore if $f_i,i=1,2$ are two such solutions to the system \eqref{sys_kin} with initial data $f_i(0)$ satisfying \eqref{main_ass}, we have
\bq\label{main_stab}
d_1(f_1(t),f_2(t)) \leq d_1(f_1(0),f_2(0))e^{Ct}, \quad \mbox{for} \quad t \in [0,T],
\eq
where $C$ is a positive constant that depends only on the $L^\infty(\R^d \times \R^d \times (0,T))$ norm of $f_1$.
\end{theorem}
We will show the existence of weak solutions satisfying the quantitative stability estimate \eqref{main_stab} in Section 4.

\begin{remark}\label{rmk_new2} Theorem {\rm\ref{thm_weak}} can be also proved by taking into account slightly weaker assumptions on $K$. More precisely, 
we may replace $(\bf{H2})$-$(iv)$ by
\begin{itemize}
\item $ ({\bf H2})$-$(iv)':$
\[
\wpa K(v) \subset \Theta(w)^{C|v-w|,+} \quad \mbox{for} \;  v,w \in \R^d
\]
where $C >0$ is independent of $v$, and $w$.
\end{itemize}
Moreover, let us mention here that $(\bf{H2})$-$(i)$ and $(iii)$ are direct consequences of the condition
\[
 \lt( x + K(v) \rt) \Delta K(w) \subset \Theta(w)^{C|v-w|,+} \quad \mbox{for}
 \;   x,v,w \in \R^d\,,
\]
leading to an easier verifiable condition for the sensitivity regions.
\end{remark}

To give an idea of the strategy of the proof, we provide the following proposition that includes the key argument of the weak-strong stability in $d_1$ and for which the assumptions on the sensitivity regions $K(v)$ are tailor-made. 

Let us point out that the classical arguments for the mean-field limit \cite{dobru,BH,Neun,Spohn,Szni,Glose} and their variants \cite{HL,BCC,CCR1,CCR2} using well-posedness of the associated Vlasov-like equations for  measures initial data are doomed to failure for sharp sensitivity regions due to the lack of uniqueness of the associated differential inclusion system as seen in Section \ref{sec_23} above. 
The weak-strong stability in $d_1$ replaces the argument of continuity with respect to initial data both for the well-posednes of the kinetic equation and the mean-field limit as already remarked for the Vlasov-Poisson equation in one dimension in \cite{Hauray}.

\begin{proposition}\label{prop:Strat}
Let $f$ be a solution to \eqref{sys_kin} given by Theorem {\rm\ref{thm_weak}}. Then the force field $F(f)$ generated by $f$ is linearly growing at infinity and locally Lipschitz continuous in phase space uniformly on $[0,T]$. More precisely, there exists a constant $C>0$ which depends on $\|f\|_{L^1\cap L^{\infty}}$ and support of $f_0$ in velocity such that  
\[
F(f)(x,v,t)\leq C(1+|v|)\|f\|_{L^1} \quad \mbox{and} \quad|F(f)(x,v) - F(f)(\tilde x, \tilde v)| \leq C(1 + |v|)|(x,v) - (\tilde x, \tilde v)|,
\]
for all $x,v,\tilde{x},\tilde{v}\in \R^d$ and $t\in [0,T]$.
\end{proposition}
\begin{proof}
First, it is clear to find 
$$\begin{aligned}
F(f)(x,v,t)&\leq \int_{\R^{2d}}|w-v|f(y,w,t)dydw \leq \int_{\R^{2d}}|w|f(y,w,t)dydw+|v|\int_{\R^{2d}}f(y,w,t)dydw\\
&\leq C(1+|v|)\|f\|_{L^1},
\end{aligned}$$	
where we used the compact support of $f$ in velocity. For the local Lipschitz continuity of the force with respect to $x$ and $v$, we estimate
$$\begin{aligned}
&F(f)(x,v) - F(f)(\tilde x, \tilde v)\cr
&= \int_{\R^d \times \R^d} \mb_{K(v)}(x-y)(w - v)f(y,w)\,dydw - \int_{\R^d \times \R^d} \mb_{K(\tilde v)}(\tilde x-y)(w - \tilde v)f(y,w)\,dydw\cr
&= \int_{\R^d \times \R^d} \lt( \mb_{K(v)}(x-y) - \mb_{K(\tilde v)}(x-y)\rt)(w-v)f(y,w)\,dydw\cr
&\quad + \int_{\R^d \times \R^d} \lt( \mb_{K(\tilde v)}(x-y) - \mb_{K( \tilde v)}(\tilde x-y)\rt)(w-v)f(y,w)\,dydw\cr
&\quad + \int_{\R^d \times \R^d} \mb_{K(\tilde v)}(\tilde x - y)(\tilde v - v)f(y,w)\,dydw\cr
&=: I_1 + I_2 + I_3.
\end{aligned}$$
By using that $f$ is compactly supported in velocity together with ${\bf (H2)}$-$(i)$, $(ii)$, and $(iii)$, we obtain that for $|x- \tilde x| \leq 1/2$ and $C|v-\tilde{v}|\leq 1$
\begin{align}\label{est_ii}
\begin{aligned}
I_1 &\leq C(1 + |v|)\int_{\R^d \times \R^d} \mb_{K(v) \Delta K(\tilde v)}(x-y)f(y,w)\,dydw \cr
&\leq C(1 + |v|)\int_{\R^d \times \R^d} \mb_{\Theta(v)^{C|v-\tilde{v}|,+}}(x-y)f(y,w)\,dydw \cr
&\leq C\|\rho\|_{L^\infty}(1 + |v|)|v- \tilde v|\leq C(1 + |v|)|v- \tilde v|,\cr
I_2 &\leq C(1 + |v|)\int_{\R^d \times \R^d} \mb_{\pa^{2|x- \tilde x|}K(\tilde v)}(x-y) f(y,w)\,dydw \cr
&\leq C\|\rho\|_{L^\infty}(1 + |v|)|\Theta(\tilde v)^{2|x - \tilde x|,+}|\leq C(1 + |v|)|x - \tilde x|,\cr
\end{aligned}
\end{align}
where $\rho := \int_{\R^d} f\,dv$. On the other hand, for $|x - \tilde x|\geq 1/2$ and $C|v-\tilde{v}|\geq$, we estimate $I_1$ and $I_2$ as 
$$\begin{aligned}
I_1 &\leq C(1+|v|)\|f\|_{L^1} \leq C(1+|v|)|v- \tilde v|,\cr
I_2 &\leq C(1+|v|)\|f\|_{L^1} \leq C(1+|v|)|x- \tilde x|.
\end{aligned}$$
Note that $I_3$ can be easily estimated by
\[
I_3 \leq C\|f\|_{L^1}|v - \tilde v|.
\]
Combining the above estimates, we have
\bq\label{est_lip}
|F(f)(x,v) - F(f)(\tilde x, \tilde v)| \leq C(1 + |v|)|(x,v) - (\tilde x, \tilde v)|.
\eq
This completes the proof.
\end{proof}
By Proposition \ref{prop:Strat}, we obtain that the flow $\mmz$ given in \eqref{traj1} is well-defined for 
weak solutions in the sense of Definition \ref{def_weak}. 
Indeed, we can define flows $\mmz(t;s,x,v) := (X(t;s,x,v),V(t;s,x,v)) : \R_+ \times \R_+ \times \R^d \times \R^d \to \R^d \times \R^d$ generated from \eqref{sys_kin} satisfying
\begin{equation}\label{traj1}
\left\{ \begin{array}{ll}
\displaystyle \frac{d}{dt} X(t;s,x,v) = V(t;s,x,v), &\\[4mm]
\displaystyle \frac{d}{dt} V(t;s,x,v) = F(f)(t,\mmz(t;s,x,v)), &
\\[4mm]
(X(s;s,x,v),V(s;s,x,v)) = (x,v), &
\end{array} \right.
\end{equation}
for all $s,t \in [0,T]$.
\subsection{Particles to Continuum: Mean-Field Limit}

The existence of a quantitative bound in $d_1$-distance where only the $L^\infty$-bound of one of the solutions is needed is a strong indication that the mean-field limit from
particles to continuum descriptions can be carried over. We define the empirical measure $\mu^N(t)$ associated to a solution to the differential inclusion system \eqref{sys_di} as
\bq\label{def_emp}
\mu^N(t) = \sum_{i=1}^N m_i \delta_{X^N_i(t),V^N_i(t)}.
\eq
The additional work to take care concerns the control of the error term in $d_1$ between weak solutions and empirical measures associated to differential inclusions. The main result of this paper can be summarized as:

\begin{theorem} \label{main}
Suppose that the sensitivity region set-valued function $K$ satisfies ${\bf (H1)}$ and ${\bf (H2)}$, and let $f$ be a weak solution to the Cauchy problem \eqref{sys_kin}-\eqref{ini_sys_kin} in the sense of Definition \ref{def_weak} up to time $T>0$ with initial data $f_0 \in (\mathcal{P}_1 \cap L^\infty)(\R^d \times \R^d)$ compactly supported in velocity. Then we have:
\[
d_1(f(t),\mu^N(t))\le e^{Ct} d_1(f(0),\mu^N(0)) \quad \mbox{for} \quad t \in [0,T],
\]
where $C$ is a positive constant that depends only on the $L^\infty(\R^d \times \R^d\times (0,T))$ norm of $f$.
\end{theorem}

It is worth emphazing that our result can not be obtained in other Wasserstein distances $d_p$, with $p>1$, see \cite{Vil} for its definition. We will explain this in details in Remark \ref{req:R2} based on technical considerations about its proof.

%
%
%
%
\section{Mean-field limit: Proof}
In this section, we give the proof of the mean-field limit from the $N$-particle system \eqref{sys_di} to the kinetic equation \eqref{sys_kin} as stated in Theorem \ref{main}. For a fixed $t _0 \in [0,T]$, we choose an optimal transport map for $d_1$ denoted by $\mt^0 =(\mt^0_1,\mt^0_2)$ between $f(t_0)$ and $\mu^N(t_0)$, i.e., $\mu^N(t_0) = \mt^0 \# f(t_0)$. Then we can find that $f(t) = \mmz(t;t_0,\cdot,\cdot) \# f(t_0)$, and it follows from the definition of the solutions to the differential inclusion system $\mu^N(t) = \mmz^N(t;t_0,\cdot,\cdot) \# \mu^N(t_0)$ for $t \geq t_0$, where $\mmz^N$ is a solution to the system \eqref{sys_di} with $\mmz^N(t_0,t_0,x,v) = (x,v)$. Moreover, we obtain
\[
\mt^t \# f(t) = \mu^N(t) \quad \mbox{where} \quad \mt^t = \mmz^N(t;t_0,\cdot,\cdot) \circ \mt^0 \circ \mmz(t_0;t,\cdot,\cdot), 
\]
for $t \in[t_0, T]$.

We then use the Wasserstein $1$-distance to get
\[
d_1(f,\mu^N) \leq \int_{\R^d \times \R^d} |\mmz(t;t_0,x,v) - \mmz^N(t;t_0,\mt^0(x,v))|f(x,v,t_0)\,dxdv.
\]
Set 
\[
Q(t) := \int_{\R^d \times \R^d} |\mmz(t;t_0,x,v) - \mmz^N(t;t_0,\mt^0(x,v))|f(x,v,t_0)\,dxdv.
\]
By the definition of solutions to the differential inclusion system, we find that the solution $\mmz^N$ is differentiable at almost every $t$ with $\mmz^N \in \mmb^N(\mmz^N(t))$. Thus we can compute the derivative of $Q$ with respect to almost every time as follows.
\begin{align*}
\begin{aligned}
\frac{d}{dt}Q(t) \Big|_{t = t_0+} &\leq \int_{\R^d \times \R^d} |V(t;t_0,x,v) - V^N(t;t_0,\mt^0(x,v))|f(x,v,t_0)dxdv \Big|_{t = t_0+} \cr
&\quad  + \int_{\R^d \times \R^d} \left| F(f)(\mmz(t;t_0,x,v),t) -F^N(\mu^N)(\mmz^N(t;t_0,\mt^0(x,v)),t) \right|f(x,v,t_0)dx dv \bigg|_{t = t_0+} \cr
&=: I + J.
\end{aligned}
\end{align*}
$\diamond$ Estimate of $I$: We clearly obtain
\[
I = \int_{\R^d \times \R^d} |v - \mt^0_2(x,v)|f(x,v,t_0)\,dxdv \leq d_1(f(t_0),\mu^N(t_0)).
\]
$\diamond$ Estimate of $J$: We notice that
\begin{align*}
\begin{aligned}
J &= \int_{\R^d \times \R^d} \bigg| \int_{\R^d \times \R^d} \mb_{K(v)}(x - y)(w - v)f(y,w,t_0)dy dw \cr
&\quad -\int_{\R^d \times \R^d} \alpha(\mt^0(x,v), y)(w - \mt^0_2(x,v))d\mu^N(y,w,t_0)
\bigg|f(x,v,t_0)dx dv \cr
&= \int_{\R^d \times \R^d} \bigg| \int_{\R^d \times \R^d} \mb_{K(v)}(x - y)(w - v)f(y,w,t_0)dy dw \cr
&\quad -\int_{\R^d \times \R^d} \alpha(\mt^0(x,v), \mt^0_1(y,w))(\mt^0_2(y,w) - \mt^0_2(x,v))f
(y,w,t_0)dy dw \bigg|f(x,v,t_0)dx dv,
\end{aligned}
\end{align*}
where $\alpha (\mt^0(x,v), \mt^0_1(y,w)) \in \mathcal{A}(\mt^0_1(x,v) - \mt^0_1(y,w),\mt^0_2(x,v))$ and given by
\begin{displaymath}
\alpha(\mt^0(x,v),\mt^0_1(y,w)) = \left\{ \begin{array}{ll}
\mb_{K(\mt^0_2(x,v))}(\mt^0_1(x,v) - \mt^0_1(y,w)) & \textrm{for } \,\,\mt^0_1(x,v) - \mt^0_1(y,w) \notin \wpa K(\mt^0_2(x,v)),\\ [1mm]
\alpha\in [0,1] & \textrm{for } \,\,\mt^0_1(x,v) - \mt^0_1(y,w) \in \wpa K(\mt^0_2(x,v)).
  \end{array} \right.
\end{displaymath}
For notational simplicity, we omit the time dependency $t_0$ in the following computations for $J$. We divide it into three terms as follows:
\[
J = \int_{\R^d \times \R^d}\lt| \sum_{i = 1}^3 J_i \rt|f(x,v)\,dxdv,
\]
where
\begin{align*}
\begin{aligned}
J_1 &= \int_{\R^d \times \R^d} \lt( \mb_{K(v)}(x - y) -\mb_{K(\mt^0_2(x,v))}(\mt^0_1(x,v) - \mt^0_1(y,w)) \rt) (w-v)f(y,w)\,dydw,\cr
J_2 &= \int_{\R^d \times \R^d} (\mb_{K(\mt^0_2(x,v))}(\mt^0_1(x,v) - \mt^0_1(y,w))-\alpha(\mt^0(x,v), \mt^0_1(y,w)))\lt(w - v\rt)f(y,w)\,dydw, 
\end{aligned}
\end{align*}
and
$$
J_3 = \int_{\R^d \times \R^d} \alpha(\mt^0(x,v), \mt^0_1(y,w))\lt( w - v - (\mt^0_2(y,w) - \mt^0_2(x,v))\rt)\,f(y,w)\,dydw.
$$
	
$\diamond$ (Estimate of $J_1$): We decompose $J_1$ into two terms:
$$\begin{aligned}
J_1 &= \int_{\R^d \times \R^d} \lt( \mb_{K(v)}(x-y) - \mb_{K(v)}(\mt^0_1(x,v) - \mt^0_1(y,w))\rt)(w-v)f(y,w)\,dydw\cr
&\quad + \int_{\R^d \times \R^d} \lt( \mb_{K(v)}(\mt^0_1(x,v) - \mt^0_1(y,w)) - \mb_{K(\mt^0_2(x,v))}(\mt^0_1(x,v) - \mt^0_1(y,w))\rt)(w-v)f(y,w)\,dydw\cr
&=:J_1^1 + J_1^2.
\end{aligned}$$
Then using compact support in velocity and Lemma \ref{lem_est2} we find
$$|J_1^1|\leq \int \lt(\mb_{\partial^{2|x-\mt^0_1(x,v)|} K(v) }(x-y)+\mb_{\partial^{2|y-\mt^0_1(y,w)|} K(v) }(x-y)\rt)f(y,w)\,dydw .
$$	
Therefore, taking into account ${\bf(H2)}$-$(ii)$, we deduce that
$$\begin{aligned}
\int_{\R^{2d}}|J_1^1|f(x,v)dxdv \leq \, & \int\mb_{\partial^{2|x-\mt^0_1(x,v)|} K(v) }(x-y)f(y,w)\,f(x,v) \,dydw\,dxdv\\
&+\int\mb_{\partial^{2|y-\mt^0_1(y,w)|} K(v) }(x-y)f(x,v)\,f(y,w) \,dxdv\,dydw\\
\leq &\,C(1+\left\| f \right\|_{\infty} )\int |x-\mt^0_1(x,v)|f(x,v)\,dxdv.
\end{aligned}
$$

Now, we proceed with the term $J_1^2$ by using ${\bf(H2)}$-$(iii)$ and the fact that for any $K\subset\R^d$ and $x,y\in \R^d$ 
$$
\mb_{K}(x)\leq \mb_{K^{|x-y|,+}}(y)
$$
since $x\in K$ implies that $y\in K^{|x-y|,+}$, to find 
\begin{align*}
\lt| \mb_{K(v)}(\mt^0_1(x,v) - \mt^0_1(y,w)) - \right. &\left.\mb_{K(\mt^0_2(x,v))}(\mt^0_1(x,v) - \mt^0_1(y,w))\rt|\\
&\leq \mb_{K(v) \Delta K(\mt^0_2(x,v))}  (\mt^0_1(x,v) - \mt^0_1(y,w))\\
&\leq \mb_{\Theta(v)^{|v-\mt^0_2(x,v)|,+}}  (\mt^0_1(x,v) - \mt^0_1(y,w))\\
&\leq \mb_{ \Theta(v)^{C|v-\mt^0_2(x,v)|+|x-\mt^0_1(x,v)|+|y-\mt^0_1(y,w)|,+}}(x-y)\\
&\leq  \mb_{ \Theta(v)^{2|x-\mt^0_1(x,v)|+2C|v-\mt_{2}^0(x,v)|,+} } (x-y)+\mb_{ \Theta(v)^{2|y-\mt^0_1(y,w)|,+}  } (x-y).
\end{align*}	
And thus using assumption $(\textbf{H2})$-$(ii)$ we conclude
$$
\int_{\R^{2d}}|J^2_1|f(x,v)\,dxdv \leq C(1+\left\|f \right\|_{\infty} )\int \lt(|x-\mt^0_1(x,v)|+|v-\mt^0_2(x,v)| \rt) f(x,v)\, dxdv.
$$
	
$\diamond$ Estimate of $J_2$:
Note that by the definition of $\alpha(\mt^0(x,v), \mt^0_1(y,w))$, we infer that
$$\mb_{K(\mt^0_2(x,v))}(\mt^0_1(x,v) - \mt^0_1(y,w))-\alpha(\mt^0(x,v), \mt^0_1(y,w))\neq 0,
$$
only if 
$$
\mt^0_1(x,v) - \mt^0_1(y,w)\in\widetilde{\partial}K(\mt_2^0(x,v))\subset \Theta(\mt_2^0(x,v)),
$$
due to Lemma~\ref{rem:BsubThe}. Therefore we deduce from $(\textbf{H2})$-$(iv)$, or equivalently by using $(\textbf{H2})$-$(iv)'$ in Remark~\ref{rmk_new2}, that
$$
x-y\in \Theta(v)^{C|v-\mt_2^0(x,v)|+|x-\mt_1^0(x,v)|+|y-\mt_2^0(x,v)|,+}.
$$
Hence 
\begin{align*}
|\mb_{K(\mt^0_2(x,v))}(\mt^0_1(x,v) - &\mt^0_1(y,w))-\alpha(\mt^0(x,v), \mt^0_1(y,w))|\\
&\leq \mb_{\Theta(v)^{C|v-\mt_2^0(x,v)|+|x-\mt_1^0(x,v)|+|y-\mt_1^0(y,w)|,+}}(x-y)\\
& \leq  
\mb_{\Theta(v)^{2C|v-\mt_2^0(x,v)|+2|x-\mt_1^0(x,v)|,+}}(x-y)+ \mb_{\Theta(v)^{2|y-\mt_1^0(y,w)|,+}}(x-y).
\end{align*}

Combining the above estimates and being $t_0$ arbitrary in $[0,T]$, we conclude
\[
\frac{d^+}{dt}d_1(f(t),\mu^N(t)) \leq Cd_1(f(t),\mu^N(t)),
\]
for all $t \in [0,T]$, where $C$ is positive constant independent of $N$. This completes the proof.

\begin{remark}\label{req:R2}
We are now able to explain the reason we are not able to consider other Wasserstein distances $d_p$, with $p>1$. To estimate the $d_p$ distance, we need to control
\[
Q_p(t) := \int_{\R^d \times \R^d} |\Xi(t;t_0,x,v)|^p f(x,v,t_0)\,dxdv \quad \mbox{with } \Xi(t;t_0,x,v)=\mmz(t;t_0,x,v) - \mmz^N(t;t_0,\mt^0(x,v)).
\]
To control this quantity by the classical use of Gronwall's inequality, we need to estimate 
$$\begin{aligned}
p\int_{\R^d \times \R^d} |V(t;t_0,x,v) - V^N(t;t_0,\mt^0(x,v))| \left| \Xi(t;t_0,x,v) \right|^{p-1}f(x,v)dx dv
\end{aligned}$$
and
$$\begin{aligned}
p\int_{\R^d \times \R^d}  &\left| F(f)(\mmz(t;t_0,x,v),t)-F^N(f)(\mmz^N(t;t_0,\mt^0(x,v)),t) \right| \left| \Xi(t;t_0,x,v) \right|^{p-1}f(x,v)dx dv.
\end{aligned}$$
Following the idea of the proof above, we would have to estimate (we drop all reference to $t_0$)
$$\begin{aligned}
\int_{\R^d \times \R^d} \!\!\!\!\!\! |\Xi(x,v)|^{p-1} \left| \int_{\R^d \times \R^d}\!\!\!\! \lt(\mb_{K(v)}(x-y) - \mb_{K(v)}(\mt_1^0(x,v)-\mt_1^0(y,w)) \rt)(w-v)\,f(y,w)\,dydw \right| f(x,v)dx dv.
\end{aligned}$$
It is not possible to use Fubini's theorem and perform an integration with respect to $x$ in the second characteristic of the set as before. Moreover, the use of any H\"older inequalities will not lead to the requested linear estimate at the end. Thus our arguments fail to estimate $d_p$ distances, $p>1$.
\end{remark}

%
%
%
%

\section{Local-in-time existence and uniqueness  of weak solutions}
In this section, we study the local existence and uniqueness of weak solutions to the system \eqref{sys_kin}.

\subsection{Regularization}
In this part, we introduce a regularized system, and show the uniform boundedness of solutions to the regularized system in regularization parameters.

We consider a mollified indicator function with respect to phase space $\mb_{K}^{\eta,\e}$ as $\mb_K^{\eta,\e} := \mb_K *_{(x,v)} (\psi_\e,\phi_\eta)$, i.e.,
\[
\mb^{\eta,\e}_{K(v)}(x) = \int_{\R^d} \mb_{K(v-w)}(x-y)\phi_\eta(w)\psi_\e(y)\,dydw,
\]
where $\phi_\eta(w) := \frac{1}{\eta^d}\phi\left( \frac{w}{\eta}\right)$ with
\[
\phi(v) = \phi(-v) \geq 0, \quad \phi \in \mc^\infty_0(\R^d), \quad \mbox{supp } \phi \subset B(0,1), \quad \mbox{and} \quad \int_{\R^d} \phi(v)\,dv = 1.
\]
Here $B(0,r):= \{x \in \R^d : |x| < r\}$ for $r > 0$. Similarly, the mollifier function $\psi_\e(y)$ is also defined. Classical theory implies that there exists a unique global solution $f^{\eta,\e}$ compactly supported in velocity to the regularized system:
\begin{equation}\label{reg-k-CS-si}
\left\{ \begin{array}{ll}
\partial_t f^{\eta,\e} + v \cdot \nabla_x f^{\eta,\e}
 + \nabla_v \cdot \big[F^{\eta,\e}(f^{\eta,\e})f^{\eta,\e}\big] = 0, & \,\, (x, v) \in \R^d \times \R^d,~~t > 0,\\[2mm]
\displaystyle F^{\eta,\e}(f^{\eta,\e})(x,v,t) := \int_{\R^d \times \R^d} \mb_{K(v)}^{\eta,\e}(x-y)(w - v)f^{\eta,\e}(y,w,t)dydw, & \,\, (x, v) \in \R^d \times \R^d,~~t > 0,\\[4mm]
f^{\eta,\e}(x,v,0) =: f_0(x,v), & \,\, (x,v) \in \R^d \times \R^d,
\end{array} \right.
\end{equation}
given by the initial data pushed forward through the characteristic system, see for instance \cite{CCR2}.

Before we proceed, we first remark the uniform boundedness, with respect to the regularization parameter, of the velocity support of $f^{\eta,\e}$ using the method of characteristics. More precisely, we consider the forward characteristics $Z^{\eta,\e}(s) := \left( X^{\eta,\e}(s;0,x,v), V^{\eta,\e}(s;0,x,v)\right)$ satisfying the following ODE system:
\begin{align}\label{traj-vel-est}
\begin{aligned}
\frac{dX^{\eta,\e}(s)}{ds} &= V^{\eta,\e}(s),\cr
\frac{dV^{\eta,\e}(s)}{ds} &= \int_{\R^d \times \R^d} \mb_{K(V^{\eta,\e}(s))}^{\eta,\e}\left( X^{\eta,\e}(s) - y\right)(w - V^{\eta,\e}(s))f^{\eta,\e}(y,w,s) dy dw.
\end{aligned}
\end{align}
Set $\Omega^{\eta,\e}(t)$ and $R^{\eta,\e}_v(t)$ the $v$-projection of compact supp$f^{\eta,\e}(\cdot,t)$ and maximum value of $v$ in $\Omega^{\eta,\e}(t)$, respectively:
\begin{equation*}
\Omega^{\eta,\e}(t) := \overline{\{ v \in \R^d : \exists (x,v) \in \R^d \times \R^d \mbox{ such that } f^{\eta,\e}(x,v,t) \neq 0 \}}, \quad R^{\eta,\e}_v(t):= \max_{v \in \Omega^{\eta,\e}(t)}|v|.
\end{equation*}
Then we have the uniform boundedness of support of $f^{\eta,\e}$ in velocity as follows. 

\begin{lemma}\label{lem-gro-vel} Let $Z^{\eta,\e}(t)$ be the solution to the particle trajectory \eqref{traj-vel-est} issued from the compact supp$_{(x,v)}f_0$ at time $0$. Then we have
\[
R^{\eta,\e}_v(t) \leq R^{\eta,\e}_v(0) = R^0_v:= \max_{v\in \Omega(0)}|v|,
\]
i.e., given $\tilde{\Omega}_0 := B(0,R^0_v) \subset \R^d$, then $\Omega^{\eta,\e}(t) \subset \tilde{\Omega}_0$ for $t \geq 0$, $\epsilon,\eta >0$.
\end{lemma}

The proof was given in \cite{CFRT} and \cite[Lemma 3.1]{CCH2}.

\begin{remark}\label{rmk-1-mom} It follows from Lemma \ref{lem-gro-vel} that all velocity moments of $f^{\eta,\e}$ and all space moments of $f^{\eta,\e}$ are uniformly bounded in $[0,T]$ with respect to the regularization parameters $\e,\eta$. 
\end{remark}

\begin{proposition}\label{prop-lp} Let $f^{\eta,\e}$ be the solution to the system \eqref{reg-k-CS-si}. Then there exists a positive time $T>0$ such that the $L^1\cap L^\infty$-estimate of $f^{\eta,\e}$:
\[
\sup_{0\leq t \leq T}\|f^{\eta,\e}(\cdot,\cdot,t)\|_{L^1 \cap L^\infty} \leq C,
\]
holds, where $C$ is a positive constant independent of $\e$ and $\eta$.
\end{proposition}
\begin{proof}
This is essentially a consequence of the local Lipschitz character of the velocity field \eqref{est_ii} observed in Section 2. Due to conservation of mass, we only need to deduce 
the $L^\infty$-estimate of $f^{\eta,\e}$, we rewrite the system \eqref{reg-k-CS-si} in the non-conservative form:
\[
\pa_t f^{\eta,\e} + v \cdot \nabla_x f^{\eta,\e} + F^{\eta,\e}(f^{\eta,\e})\cdot\nabla_v f^{\eta,\e} = -f^{\eta,\e} \nabla_v \cdot (F^{\eta,\e}(f^{\eta,\e}))
\]
Note that
$$\begin{aligned}
\nabla_v \cdot (F^{\eta,\e}(f^{\eta,\e})) &= -d\int_{\R^d \times \R^d} \mb^{\eta,\e}_{K(v)}(x-y)f^{\eta,\e}(y,w)\,dydw \cr
&\quad + \int_{\R^d \times \R^d} \lt(\nabla_v \mb^{\eta,\e}_{K(v)}(x-y) \rt)\cdot(w-v)f^{\eta,\e}(y,w)\,dydw.
\end{aligned}$$
The first term is obviously bounded by the $L^1$ norm of $f_0$. In order to estimate the second term in the right hand side of the above equation, we find that for $i = 1,\cdots, d$ and $h > 0$
$$\begin{aligned}
&\int_{\R^d \times \R^d} \lt( \mb_{K(v + e_i h)}^{\eta,\e}(x-y) - \mb_{K(v)}^{\eta,\e}(x-y)\rt) (w_i-v_i)f^{\eta,\e}(y,w)\,dydw\cr
&\,\, = \int_{\R^{2d} \times \R^{2d}} \lt( \mb_{K(v + e_i h - v')}(x-y - x') - \mb_{K(v-v')}(x-y - x')\rt) (w_i-v_i)f^{\eta,\e}(y,w)\phi_\eta(v')\psi_\e(x')\,dydwdx'dv'\cr
&\,\, \leq C\|f^{\eta,\e}\|_{L^\infty}\int_{\R^d \times \R^d} \lt( \int_{\R^d} \lt|\mb_{K(v + e_i h - v')}(x-y - x') - \mb_{K(v-v')}(x-y - x')\rt| dy\rt)\phi_\eta(v')\psi_\e(x')\,dx'dv'\cr
&\,\, \leq C\|f^{\eta,\e}\|_{L^\infty}h,
\end{aligned}$$
where we used Lemma \ref{lem-gro-vel} and a similar argument to \eqref{est_ii}.
This yields that
\[
\|\nabla_v \cdot (F^{\eta,\e}(f^{\eta,\e}))\|_{L^\infty} \leq C\|f^{\eta,\e}\|_{L^1 \cap L^\infty}
\]
and
\[
\frac{d}{ds}f^{\eta,\e}(X^{\eta,\e}(s),V^{\eta,\e}(s),s) \leq C(1 + \|f^{\eta,\e}\|_{L^\infty})f^{\eta,\e}(X^{\eta,\e}(s),V^{\eta,\e}(s),s).
\]
We conclude that
\[
\frac{d}{dt}\|f^{\eta,\e}(\cdot,\cdot,t)\|_{L^1 \cap L^\infty} \leq \|f^{\eta,\e}(\cdot,\cdot,t)\|_{L^1 \cap L^\infty}^2
\]
Hence we conclude that there exists a finite time $T > 0$ such that
\[
\sup_{0 \leq t \leq T}\|f^{\eta,\e}(\cdot,\cdot,t)\|_{L^\infty} \leq C,
\]
where $C$ is a positive constant independent of $\e$ and $\eta$.
\end{proof}

In the following lemma, we provide elementary properties of the mollified indicator function and several useful estimates of the differences between mollified indicator functions.
$\mb_{K}^{\e}$ denotes the mollified indicator function with respect to $x$, i.e., $\mb_{K}^{\e} = \mb_{K} *_x \phi_{\e}$. Similarly, $\mb_{K}^{\eta}$ denotes the mollified indicator function with respect to $v$, i.e., $\mb_{K}^{\eta} = \mb_{K} *_v \phi_{\eta}$.

\begin{lemma} 
\

\mbox{(i)} For any $\e > 0$ and compact set $K \subset \R^d$, we have
\bq\label{lem_basic1}
\mb_{K}^\e(x) \leq \mb_{K^{\e,+}}(x).
\eq

(ii) For all $v \in \R^d$ and all $ \e > 0$, it holds
\bq\label{lem_diff1}
\int_{\R^d}\lt|\mb_{K(v)}^{\eta,\e}(x) - \mb_{K(v)}^{\eta}(x)\rt| dx \leq \sup_{v \in \R^d}|\pa^{2\e}K(v)|.
\eq

(iii) For all $v \in \R^d$ and $0 < \eta \leq 1$, it holds
\bq\label{lem_add}
\int_{\R^d}\lt|\mb_{K(v)}^{\eta,\e}(x) - \mb_{K(v)}^{\e}(x)\rt| dx \leq C\eta,
\eq
where $C$ is a positive constant independent of $\e$ and $\eta$.\newline

(iv) For all $\eta, \e > 0$ and $x_1,y_1,x_2,y_2,v \in \R^d$ we have
\bq\label{lem_diff2}
|\mb_{K(v)}^{\eta,\e}(y_1 - x_1) - \mb_{K(v)}^{\eta,\e}(y_2 - x_2)| \leq \mb^{\eta,\e}_{\pa^{2|x_1 - x_2|}K(v)}(y_1 - x_1) + \mb^{\eta,\e}_{\pa^{2|y_1 - y_2|}K(v)}(y_1 - x_1).
\eq
\end{lemma}
\begin{proof} 
\

\noindent (i) By the definition of the mollification, it is clear to get the desired inequality. 

\noindent (ii) A straightforward computation yields 
$$\begin{aligned}
|\mb_{K(v)}^{\eta,\e}(x) - \mb_{K(v)}^{\eta}(x)| &\leq \int_{\R^d \times \R^d} |\mb_{K(w)}(x-y) - \mb_{K(w)}(x)|\phi_\eta(v-w)\psi_\e(y)\,dydw\cr
& = \int_{\{y \in \R^d : |y| \leq \e\} \times \R^d} |\mb_{K(w)}(x-y) - \mb_{K(w)}(x)|\phi_\eta(v-w)\psi_\e(y)\,dydw\cr
&\leq \int_{\R^d \times \R^d} \mb_{\pa^{2\e}K(w)}(x)\phi_\eta(v-w)\psi_\e(y)\,dydw\cr
&=\mb_{\pa^{2\e}K(v)}^\eta(x).
\end{aligned}$$
This concludes
$$\begin{aligned}
\int_{\R^d} |\mb_{K(v)}^{\eta,\e}(x) - \mb_{K(v)}^{\eta}(x)|dx &\leq \int_{\R^d}\mb_{\pa^{2\e}K(v)}^\eta(x)\,dx = \int_{\R^d \times \R^d} |\pa^{2\e}K(w)|\phi_\eta(v-w)\,dwdx\cr
&\leq \sup_{v \in \R^d}  |\pa^{2\e}K(v)|.
\end{aligned}$$
(iii) In a similar fashion as above, we find
$$\begin{aligned}
|\mb_{K(v)}^{\eta,\e}(x) - \mb_{K(v)}^{\e}(x)| &\leq \int_{\R^d \times \R^d} |\mb_{K(v-w)}(x-y) - \mb_{K(v)}(x-y)|\phi_\eta(w)\psi_\e(y)\,dydw\cr
& = \int_{\R^d \times \{w \in \R^d : |w| \leq \eta\} } |\mb_{K(v-w)}(x-y) - \mb_{K(v)}(x-y)|\phi_\eta(v-w)\psi_\e(y)\,dydw.
\end{aligned}$$
This implies that
$$\begin{aligned}
&\int_{\R^d}\lt|\mb_{K(v)}^{\eta,\e}(x) - \mb_{K(v)}^{0,\e}(x)\rt| dx\cr
&\quad \leq \int_{\R^d \times \{w \in \R^d : |w| \leq \eta\} } \lt(\int_{\R^d}|\mb_{K(v-w)}(x-y) - \mb_{K(v)}(x-y)|dx\rt)\phi_\eta(v-w)\psi_\e(y)\,dydw\cr
&\quad \leq \int_{\R^d \times \{w \in \R^d : |w| \leq \eta\} } |K(v-w)\Delta K(v)|\phi_\eta(v-w)\psi_\e(y)\,dydw\cr
&\quad \leq C\eta,
\end{aligned}$$
where $C$ is a positive constant independent of $\e$ and $\eta$, due to ${\bf (H2)}$-$(ii)$ and $(iii)$. 

\noindent (iv) As a direct consequence of Lemma \ref{lem_est2}, we obtain
$$\begin{aligned}
&|\mb_{K(v)}^{\eta,\e}(y_1 - x_1) - \mb_{K(v)}^{\eta,\e}(y_2 - x_2)|\cr
&\quad \leq \int_{\R^d \times \R^d} \lt| \mb_{K(v-w)}(y_1 - x_1 - z) - \mb_{K(v-w)}(y_2 - x_2 - z)\rt|\psi_\e(z)\phi_\eta(w)\,dzdw\cr
&\quad \leq \int_{\R^d \times \R^d} \mb_{\pa^{2|x_1 - x_2|}K(v-w)}(y_1 - x_1 -z)\psi_\e(z)\phi_\eta(w)\,dzdw \cr
&\qquad + \int_{\R^d \times \R^d} \mb_{\pa^{2|y_1 - y_2|}K(v-w)}(y_1 - x_1 -z)\psi_\e(z)\phi_\eta(w)\,dzdw\cr
&\quad = \mb^{\eta,\e}_{\pa^{2|x_1 - x_2|}K(v)}(y_1 - x_1) + \mb^{\eta,\e}_{\pa^{2|y_1 - y_2|}K(v)}(y_1 - x_1),
\end{aligned}$$
leading to the desired estimate.
\end{proof}

We now show the growth estimate of $d_1(f^{\eta,\e}(t),f^{\eta',\e'}(t))$ leading to a weak-strong stability which is the core of this section.

\begin{proposition}\label{prop-grow} Let $f^{\eta,\e}$ and $f^{\eta',\e'}$ be solutions to the system \eqref{reg-k-CS-si}. Then we have
\[
d_1(f^{\eta,\e}(t),f^{\eta',\e'}(t)) \leq C\lt(\eta+\e+ \eta'+\e'\rt),
\]
for $\frac12 \geq \e,\eta, \e',\eta' > 0$ with $C$ independent of $\eta$, $\e$, $\eta'$, and $\e'$, only depending on the uniform $L^\infty$-norm of the solutions ensured by Propostion \ref{prop-lp} and $T$. Moreover, given two solutions $f_1^{\eta,\e}$ and $f_2^{\eta,\e}$ to the system \eqref{reg-k-CS-si} with two different initial data, then
\[
d_1(f^{\eta,\e}_1(t),f^{\eta,\e}_2(t)) \leq C\lt(d_1(f^{\eta,\e}_1(0),f^{\eta,\e}_2(0)) + \eta+\e\rt),
\]
for $\frac12\geq \e,\eta > 0$ with $C$ independent of $\eta$, $\e$, only depending on the uniform $L^\infty$-norm of {\bf one} of the solutions, $f^{\eta,\e}_1$ for instance, and $T$.
\end{proposition}

\begin{proof} Let us consider the flows $Z^{\eta,\e} := (X^{\eta,\e},V^{\eta,\e})$ and $Z^{\eta',\e'}:= (X^{\eta',\e'},V^{\eta',\e'})$ generated from the solutions to \eqref{reg-k-CS-si} satisfying the corresponding characteristic systems \eqref{traj-vel-est} well-defined for $s,t \in [0,T]$ due to their smoothness. We now choose an optimal transport map $\mt^0 = (\mt_1^0(x,v), \mt_2^0(x,v))$ between $f^{\eta,\e}(t_0)$ and $f^{\eta',\e'}(t_0)$ for fixed $t_0 \in [0,T)$, i.e., $f^{\eta',\e'}(t_0) = \mt^0 \# f^{\eta,\e}(t_0)$. It is classically known from \cite{Vil} that such an optimal transport map exists when $f^{\eta,\e}(t_0)$ is absolutely continuous with respect to the Lebesgue measure. The solutions of these approximated problems can be expressed by the characteristic method with initial data at $t=t_0$ as $f^{\eta,\e}(t) = Z^{\eta,\e}(t;t_0,\cdot,\cdot)\#f^{\eta,\e}(t_0)$ and $f^{\eta',\e'}(t) = Z^{\eta',\e'}(t;t_0,\cdot,\cdot)\#f^{\eta',\e'}(t_0)$.  We also notice that
\[
\mt^t \# f^{\eta,\e}(t) = f^{\eta',\e'}(t), \quad \mbox{where} \quad \mt^t = Z^{\eta',\e'}(t;t_0,\cdot,\cdot) \circ \mt^0 \circ Z^{\eta,\e}(t_0;t,\cdot,\cdot).
\]
By Definition \ref{defdp} of $d_1$-distance, we obtain
\[
d_1(f^{\eta,\e}(t),f^{\eta',\e'}(t)) \leq \int_{\R^d \times \R^d} |Z^{\eta,\e}(t;t_0,x,v) - Z^{\eta',\e'}(t;t_0,\mt^0(x,v))|f^{\eta,\e}(x,v,t_0) dx dv.
\]
Set
\[
Q_{\eta,\e,\eta',\e'}(t) := \int_{\R^d \times \R^d} |Z^{\eta,\e}(t;t_0,x,v) - Z^{\eta',\e'}(t;t_0,\mt^0(x,v))|f^{\eta,\e}(x,v,t_0) dx dv.
\]
Then straightforward computations yield
\begin{align*}
\begin{aligned}
\frac{d}{dt}& Q_{\eta,\e,\eta',\e'}(t) \Big|_{t = t_0+} \cr
\leq &\int_{\R^d \times \R^d} |V^{\eta,\e}(t;t_0,x,v) - V^{\eta',\e'}(t;t_0,\mt^0(x,v))|f^{\eta,\e}(x,v,t_0)dxdv \Big|_{t = t_0+} \cr
& + \!\!\int_{\R^d \times \R^d} \left| F^{\eta,\e}(f^{\eta,\e})(Z^{\eta,\e}(t;t_0,x,v),t) -F^{\eta',\e'}(f^{\eta',\e'})(Z^{\eta',\e'}(t;t_0,\mt^0(x,v)),t) \right|f^{\eta,\e}(x,v,t_0)dx dv \bigg|_{t = t_0+} \cr
=: &\, I+J.
\end{aligned}
\end{align*}
It is easy to estimate $I$ as
\begin{equation*}
I = \int_{\R^d \times \R^d} |v - \mt^0_2(x,v)|f^{\eta,\e}(x,v,t_0) dx dv \leq C d_1(f^{\eta,\e}(t_0),f^{\eta',\e'}(t_0)).
\end{equation*}
For the estimate of $J$, we notice that
\begin{align*}
\begin{aligned}
J &= \int_{\R^d \times \R^d} \bigg| \int_{\R^d \times \R^d} \mb_{K(v)}^{\eta,\e}(x - y)(w - v)f^{\eta,\e}(y,w,t_0)dy dw \cr
&\quad -\int_{\R^d \times \R^d} \mb_{K(\mt^0_2(x,v))}^{\eta,\e'}(\mt^0_1(x,v)- y)(w - \mt^0_2(x,v))f^{\eta',\e'}
(y,w,t_0)dy dw \bigg|f^{\eta,\e}(x,v,t_0)dx dv \cr
&= \int_{\R^d \times \R^d} \bigg| \int_{\R^d \times \R^d} \mb_{K(v)}^{\eta,\e}(x - y)(w - v)f^{\eta,\e}(y,w,t_0)dy dw \cr
&\quad -\int_{\R^d \times \R^d} \mb_{K(\mt^0_2(x,v))}^{\eta,\e'}(\mt^0_1(x,v)- \mt^0_1(y,w))(\mt^0_2(y,w) - \mt^0_2(x,v))f^{\eta,\e}
(y,w,t_0)dy dw \bigg|f^{\eta,\e}(x,v,t_0)dx dv.
\end{aligned}
\end{align*}
We omit the time dependency on $t_0$ in the rest of computations for notational simplicity. We now decompose $J$ into four parts: 
\[
J = \int_{\R^d \times \R^d} \lt| \sum_{i=1}^4J_i\rt| f^{\eta,\e}(x,v)\, dx dv,
\]
where 
\begin{align*}
\begin{aligned}
J_1 &:= \int_{\R^d \times \R^d} \lt( \mb_{K(v)}^{\eta,\e}(x-y) - \mb_{K(v)}^{\eta',\e'}(x-y)\rt)(w-v)f^{\eta,\e}(y,w)\,dydw,\cr
J_2 &:= \int_{\R^d \times \R^d} \lt( \mb_{K(v)}^{\eta',\e'}(x-y) - \mb_{K(v)}^{\eta',\e'}(\mt^0_1(x,v) - \mt^0_1(y,w))\rt) (w-v) f^{\eta,\e}(y,w)\,dydw,\cr
J_3 &:= \int_{\R^d \times \R^d} \lt( \mb_{K(v)}^{\eta',\e'}(\mt^0_1(x,v) - \mt^0_1(y,w)) -\mb_{K(\mt^0_2(x,v))}^{\eta',\e'}(\mt^0_1(x,v) - \mt^0_1(y,w)) \rt) (w-v)f^{\eta,\e}(y,w)\,dydw,\cr
J_4 &:= \int_{\R^d \times \R^d} \mb_{K(\mt^0_2(x,v))}^{\eta',\e'}(\mt^0_1(x,v) - \mt^0_1(y,w))(w - v - \lt(\mt^0_2(y,w) - \mt^0_2(x,v))\rt)f^{\eta,\e}(y,w)\,dydw.
\end{aligned}
\end{align*}
$\diamond$ (Estimate of $J_1$): We decompose $J_1$ into two terms:
$$\begin{aligned}
J_1 &= \int_{\R^d \times \R^d} \lt( \mb_{K(v)}^{\eta,\e}(x-y) - \mb_{K(v)}^{\eta,\e'}(x-y)\rt)(w-v)f^{\eta,\e}(y,w)\,dydw\cr
&\quad + \int_{\R^d \times \R^d} \lt( \mb_{K(v)}^{\eta,\e'}(x-y) - \mb_{K(v)}^{\eta',\e'}(x-y)\rt)(w-v)f^{\eta,\e}(y,w)\,dydw\cr
&=:J_1^1 + J_1^2.
\end{aligned}$$
For the estimate of $J_1^1$, we obtain that for $0 < \e \leq 1/2$
$$\begin{aligned}
&\int_{\R^{2d} \times \R^{2d}} \lt| \mb_{K(v)}^{\eta,\e}(x-y) - \mb_{K(v)}^{\eta}(x-y)\rt||w-v|f^{\eta,\e}(x,v)f^{\eta,\e}(y,w)\,dxdydvdw\cr
&\quad \leq C\int_{ \R^{2d} \times \tilde\om_0^2}\lt| \mb_{K(v)}^{\eta,\e}(x-y) - \mb_{K(v)}^{\eta}(x-y)\rt|f^{\eta,\e}(x,v)f^{\eta,\e}(y,w)\,dxdydvdw\cr
&\quad \leq C\|f^{\eta,\e}\|_{L^\infty}\int_{\R^d \times \R^d}\lt(\int_{\R^d \times \tilde\om_0}\lt| \mb_{K(v)}^{\eta,\e}(x-y) - \mb_{K(v)}^{\eta}(x-y)\rt|dydw\rt)f^{\eta,\e}(x,v)\,dxdv\cr
&\quad \leq C\|f^{\eta,\e}\|_{L^\infty}\sup_{v \in \R^d} |\pa^{2\e}K(v)|\int_{\R^d \times \R^d} f^{\eta,\e}(x,v)\,dxdv\cr
&\quad \leq C\|f^{\eta,\e}\|_{L^\infty}\e,
\end{aligned}$$
due to ${\bf (H2)}$ and \eqref{lem_diff1}.
Similarly we also get
\[
\int_{\R^{2d} \times \R^{2d}} \lt| \mb_{K(v)}^{\eta,\e'}(x-y) - \mb_{K(v)}^{\eta}(x-y)\rt||w-v|f^{\eta,\e}(x,v)f^{\eta,\e}(y,w)\,dxdydvdw \leq C\|f^{\eta,\e}\|_{L^\infty}\e'
\]
for $0 < \e' \leq 1/2$. This yields
\[
\int_{\R^d \times \R^d} |J_1| f^{\eta,\e}(x,v)\,dxdv \leq C\|f^{\eta,\e}\|_{L^\infty}(\e + \e').
\]
We next use \eqref{lem_add} together with the compact support in velocity similarly to \eqref{est_ii} to find
$$\begin{aligned}
&\int_{\R^{2d} \times \R^{2d}} \lt| \mb_{K(v)}^{\eta,\e'}(x-y) - \mb_{K(v)}^{\e'}(x-y)\rt||w-v|f^{\eta,\e}(x,v)f^{\eta,\e}(y,w)\,dxdydvdw\cr
&\quad \leq C\|f^{\eta,\e}\|_{L^\infty}\int_{\R^d \times \R^d} \lt( \int_{\R^d} \lt| \mb_{K(v)}^{\eta,\e'}(x-y) - \mb_{K(v)}^{\e'}(x-y)\rt| dy\rt)f^{\eta,\e}(x,v)\,dxdv\cr
&\quad \leq C\|f^{\eta,\e}\|_{L^\infty}\eta.
\end{aligned}$$
Hence, we have
\[
\int_{\R^d \times \R^d} |J_1^2| f^{\eta,\e}(x,v)\,dxdv \leq C\|f^{\eta,\e}\|_{L^\infty}(\eta + \eta'),
\]
and this concludes
\[
\int_{\R^d \times \R^d} |J_1| f^{\eta,\e}(x,v)\,dxdv \leq C\|f^{\eta,\e}\|_{L^\infty}(\eta + \eta' + \e + \e'),
\]
where $C$ is a positive constant independent of $\eta,\eta',\e$, and $\e'$.

$\diamond$ (Estimate of $J_2$): We first use \eqref{lem_diff2} to find
	$$|J_2| \leq \int_{\R^d \times \R^d} \lt( \mb_{\pa^{2|x-\mt^0_1(x,v)|}K(v)}^{\eta',\e'}(x-y) + \mb_{\pa^{2|y- \mt^0_1(y,w)|}K(v)}^{\eta',\e'}(x-y)\rt)|w-v|f^{\eta,\e}(y,w)\,dydw$$
Moreover, by using \eqref{lem_basic1} we obtain

$$\begin{aligned}
	|J_2| &\leq \int_{\R^d \times \R^d\times \R^d} \mb_{\pa^{2|x-\mt^0_1(x,v)|+\e'}K(v-v')}(x-y) |w-v|f^{\eta,\e}(y,w)\phi_{\eta'}(v')\,dydwdv'\cr
	&\quad + \int_{\R^d \times \R^d\times \R^d} \mb_{\pa^{2|y-\mt^0_1(y,w)|+\e'}K(v-v')}(x-y) |w-v|f^{\eta,\e}(y,w)\phi_{\eta'}(v')\,dydwdv'\cr
		&=:J_2^1 + J_2^2.
\end{aligned}$$
We now choose $0 < \e' \leq 1/2$ and consider a set $D := \{ (x,v) \in \R^{2d} : |x - \mt^0_1(x,v)| \leq 1/4\}$. Then we obtain
$$\begin{aligned}
&\int_{D} |J_2^1| f^{\eta,\e}(x,v)\,dxdv \cr
&\qquad \leq C\int_{D \times \R^d \times \tilde\om_0 \times \R^d} \mb_{\pa^{2|x- \mt^0_1(x,v)| + \e'}K(v-v')} (x-y)f^{\eta,\e}(x,v)f^{\eta,\e}(y,w)\phi_{\eta'}(v')\,dxdydvdwdv'\cr
&\qquad \leq C\|f^{\eta,\e}\|_{L^\infty} \int_{\R^d \times \R^d}\sup_{u\in\R^d}\lt(|\pa^{2|x- \mt^0_1(x,v)| + \e'}K(u)|\rt)f^{\eta,\e}(x,v)\,dxdv\cr
&\qquad \leq C\|f^{\eta,\e}\|_{L^\infty}\int_{\R^d \times \R^d} \lt( |x- \mt^0_1(x,v)| + \e'\rt)f^{\eta,\e}(x,v)\,dxdv\cr
&\qquad \leq C\|f^{\eta,\e}\|_{L^\infty}\lt(\e' + d_1(f^{\eta,\e},f^{\eta',\e'})\rt),
\end{aligned}$$
due to ${\bf (H2)}$-$(ii)$. On the other hand, for $(x,v) \in D^c$, we have
$$\begin{aligned}
\int_{D^c} |J_2^1| f^{\eta,\e}(x,v)\,dxdv &\leq C\int_{D^c \times \R^{2d}} f^{\eta,\e}(x,v)f^{\eta,\e}(y,w)\,dxdydvdw\cr
&\leq C\int_{\R^{2d} \times \R^{2d}} |x - \mt^0_1(x,v)|f^{\eta,\e}(x,v)f^{\eta,\e}(y,w)\,dxdydvdw\cr
&\leq Cd_1(f^{\eta,\e},f^{\eta',\e'}).
\end{aligned}$$
This yields that for $0 < \e' \leq 1/2$
\[
\int_{\R^d \times \R^d} |J_2^1| f^{\eta,\e}(x,v)\,dxdv \leq C\|f^{\eta,\e}\|_{L^\infty}\e' + C(1 + \|f^{\eta,\e}\|_{L^\infty})d_1(f^{\eta,\e},f^{\eta,\e'}).
\]
Similarly, for $J_2^2$, we find
$$\begin{aligned}
&\int_{\R^d \times \R^d} \left(\int_{D \times \R^d\times \R^d} \mb_{\pa^{2|y-\mt^0_1(y,w)|+\e'}K(v-v')}(x-y) |w-v|f^{\eta,\e}(y,w)\phi_{\eta'}(v')\,dydwdv' \right)  f^{\eta,\e}(x,v)\,dxdv \cr
&\qquad \leq C\int_{D \times \R^d \times \tilde\om_0\times \R^d} \mb_{\pa^{2|y- \mt^0_1(y,w)| + \e'}K(v-v')} (x-y)f^{\eta,\e}(x,v)f^{\eta,\e}(y,w)\phi_{\eta'}(v')\,dxdydvdwdv'\cr
&\qquad \leq C\|f^{\eta,\e}\|_{L^\infty} \int_{\R^d \times \R^d}\sup_{u\in\R^d}\lt(|\pa^{2|y- \mt^0_1(y,w)| + \e'}K(u)|\rt)f^{\eta,\e}(y,w)\,dydw\cr
&\qquad \leq C\|f^{\eta,\e}\|_{L^\infty}\int_{\R^d \times \R^d} \lt( |y- \mt^0_1(y,w)| + \e'\rt)f^{\eta,\e}(y,w)\,dydw\cr
&\qquad \leq C\|f^{\eta,\e}\|_{L^\infty}\lt(\e' + d_1(f^{\eta,\e},f^{\eta',\e'})\rt),
\end{aligned}$$
and this and together with the estimate of $J_2^1$ implies
$$\int_{\R^d\times\R^d} |J_2|f^{\eta,\e}(x,v)\,dxdv\leq C\|f^{\eta,\e}\|_{L^\infty}\e' + C(1 + \|f^{\eta,\e}\|_{L^\infty})d_1(f^{\eta,\e},f^{\eta,\e'})$$

$\diamond$ (Estimate of $J_3$): Note that $J_3$ can be rewritten as
	\[
	\begin{aligned}
	J_3= &\int_{\R^d \times \R^d\times\R^{2d}}  \mb_{K(v-v')}(\mt^0_1(x,v) - \mt^0_1(y,w)-z)  (w-v)f^{\eta,\e}(y,w)\,\phi_{\e'}(z)\phi_{\eta'}(v')\,dydw\,dzdv'\cr
	&-\int_{\R^d \times \R^d\times\R^{2d}}  \mb_{K(\mt^0_2(x,v)-v')}(\mt^0_1(x,v) - \mt^0_1(y,w)-z)  (w-v)f^{\eta,\e}(y,w)\,\phi_{\e'}(z)\phi_{\eta'}(v')\,dydw\,dzdv'.
	\end{aligned}
	\]
	Then using assumption $(\textbf{H2})$-$(iii)$, \eqref{lem_basic1} and \eqref{lem_est1} we find
	$$\begin{aligned}
	|J_3|&\leq \int_{\R^d \times \R^d\times\R^{2d}}  \mb_{K(v-v')\Delta K(\mt^0_2(x,v)-v')}(\mt^0_1(x,v) - \mt^0_1(y,w)-z)  f^{\eta,\e}(y,w)\,\phi_{\e'}(z)\phi_{\eta'}(v')\,dydw\,dzdv'\cr
	&\leq C\int_{\R^d \times \R^d\times\R^{2d}}  \mb_{\Theta(v-v')^{C|v-\mt^0_2(x,v)|,+}}(\mt^0_1(x,v) - \mt^0_1(y,w)-z)  f^{\eta,\e}(y,w)\,\phi_{\e'}(z)\phi_{\eta'}(v')\,dydw\,dzdv'\cr
	&\leq C\int_{\R^d \times \R^d\times\R^{d}}  \mb_{\Theta(v-v')^{C|v-\mt^0_2(x,v)|+\e',+}}(\mt^0_1(x,v) - \mt^0_1(y,w))  f^{\eta,\e}(y,w)\,\phi_{\eta'}(v')\,dydw\,dv'.\cr
	\end{aligned}$$
	It follows from Lemmas \ref{lem_est1} and \ref{lem_est2} that for any compact set $K\subset \R^d$, $x,y\in\R^d$, we have
	\[
	\mb_{K}(x)\leq \mb_{K^{|x-y|,+}}(y).
	\]
Thus we deduce 
$$\begin{aligned}
|J_3|&\leq C\int_{\R^d \times \R^d\times\R^{d}}  \mb_{\Theta(v-v')^{C|v-\mt^0_2(x,v)|+|x-\mt^0_1(x,v)|+|y-\mt^0_1(y,w)|+\e',+}}(x -y)  f^{\eta,\e}(y,w)\,\phi_{\eta'}(v')\,dydwdv'\\
&\leq  C\int_{\R^d \times \R^d\times\R^{d}}  \mb_{\Theta(v-v')^{2C|v-\mt^0_2(x,v)|+2|x-\mt^0_1(x,v)|}}(x -y) f^{\eta,\e}(y,w)\,\phi_{\eta'}(v')\,dydwdv \\
&\quad +  C\int_{\R^d \times \R^d\times\R^{d}}  \mb_{\Theta(v-v')^{2|y-\mt^0_1(y,w)|+2\e',+}}(x -y)  f^{\eta,\e}(y,w)\,\phi_{\eta'}(v')\,dydwdv \\
&:=J_3^1+J_3^2.
\end{aligned}$$
Introducing a set $F:=\{(x,v)\in\R^{2d} \ , \ 2|x-\mt_1^0(x,v)|+2C|v-\mt_2^0(x,v)|\leq 1 \}$, we deduce from $(\textbf{H2})$ that
	$$\begin{aligned}
	\int_{F}|J_3^1|f^{\eta,\e}(x,v)\, dxdv &\leq C \| f^{\eta,\e} \|_{\infty} 	\int_F \sup_{u\in \R^d} |\Theta(u)^{2|x-\mt_1^0(x,v)|+2C|v-\mt^0_2(x,v)|}|f^{\eta,\e}(x,v) \, dxdv\cr
	&\leq C \| f^{\eta,\e} \|_{\infty}\int_F\lt(|x-\mt_1^0(x,v)|+ |v-\mt^0_2(x,v)|\rt)f^{\eta,\e}(x,v) \, dxdv.		
	\end{aligned}$$ 
	On the other hand, we obtain that for $(x,v) \in F^c$ 
	$$	\int_{F^c}|J_3^1|f^{\eta,\e}(x,v)\, dxdv \leq C\int_{F^c} \lt(|x-\mt_1^0(x,v)|+ |v-\mt^0_2(x,v)|\rt)f^{\eta,\e}(x,v)\,dxdv. 
	$$
	Using the similar fashion as the above, we estimate $J_3^2$ for $\e'\leq 1/4$ as 
	$$\int_{\R^{2d}}|J_3^2|f^{\eta,\e}(x,v)\,dxdv \leq C\|f^{\eta,\e}\|_{\infty}+ C\bigl(\| f^{\eta,\e}\|_{\infty}+1\bigr)\int_{\R^{2d}}|x-\mt_1^0(x,v)|f^{\eta,\e}(x,v)\,dxdv. $$
This concludes that
	\[
	\int_{\R^d \times \R^d} |J_3| f^{\eta,\e}(x,v)\,dxdv \leq C\|f^{\eta,\e}\|_{L^\infty}\e' + C(1 + \|f^{\eta,\e}\|_{L^\infty})d_1(f^{\eta,\e},f^{\eta',\e'}).
	\]
$\diamond$ (Estimate of $J_4$): Since $\mb_{K}^{\eta,\e} \leq 1$, we easily have
$$\begin{aligned}
\int_{\R^d \times \R^d} |J_4| f^{\eta,\e}(x,v)\,dxdv &\leq \int_{\R^{2d} \times \R^{2d}} \lt(|w-\mt^0_2(y,w)| + |v-\mt^0_2(x,v)|\rt)f^{\eta,\e}(x,v)f^{\eta,\e}(y,w)\,dxdydvdw\cr
&\leq 2\int_{\R^d \times \R^d} |v - \mt^0_2(x,v)| f^{\eta,\e}(x,v)\,dxdv \cr
&\leq Cd_1(f^{\eta,\e},f^{\eta',\e'}).
\end{aligned}$$
We now combine the all estimates above to have
\[
J \leq C\|f^{\eta,\e}\|_{L^\infty}(\e + \e' + \eta + \eta') + C(1 + \|f^{\eta,\e}\|_{L^\infty})d_1(f^{\eta,\e}, f^{\eta',\e'}) \quad \mbox{for} \quad \eta,\eta',\e,\e' \leq \frac12,
\]
using Proposition \ref{prop-lp}. We conclude that
\[
\frac{d^+}{dt}d_1(f^{\eta,\e}(t),f^{\eta',\e'}(t)) \leq C\lt(d_1(f^{\eta,\e}(t),f^{\eta',\e'}(t)) + \eta+\e+ \eta'+\e'\rt),
\]
for $\frac12 \geq \e,\eta, \e',\eta' > 0$ with $C$ independent of $\eta$, $\e$, $\eta'$, and $\e'$, only depending on the uniform $L^\infty$-norm of the solutions ensured by Propostion \ref{prop-lp}. Here $\frac{d^+}{dt}$ denotes the time-derivative from the right. This implies that 
\[
d_1(f^{\eta,\e}(t),f^{\eta',\e'}(t)) \leq C\lt(\eta+\e+ \eta'+\e'\rt),
\]
for $\frac12 \geq \e,\eta, \e',\eta' > 0$ with $C$ independent of $\eta$, $\e$, $\eta'$, and $\e'$.  
It is straightforward to check that in the previous argument the fact that both solutions have the same initial data does not play any role. Moreover, it is evident going through the proof that we only used the $L^\infty$-norm on one of the solutions, actually $f^{\eta,\e}$.
\end{proof}


\subsection{Passing to the limit as $\eta,\e \to 0$ and proof of Theorem \ref{thm_weak}}\label{sec3.2} It follows from Proposition \ref{prop-grow} that $\{f^{\eta,\e}\}_{\eta,\e >0}$ is a Cauchy sequence in $\mathcal{C}([0,T];\pp_1(\R^d \times \R^d))$, and this implies that there exists a limit curve of measures $f \in \mathcal{C}([0,T];\pp_1(\R^d \times \R^d))$, actually densities $f \in L^{\infty}(0,T;(L_+^1 \cap L^\infty)(\R^d \times \R^d))$ due to Proposition \ref{prop-lp}. Thus it only remains to show that $f$ is a solution of the equations \eqref{sys_kin}. Choose a test function $\Psi(x,v,t) \in \mathcal{C}_c^{\infty}(\R^d \times \R^d \times [0,T])$, then $f^{\eta,\e}$ satisfies
\begin{align}\label{weak-var}
\begin{aligned}
&\int_{\R^d \times \R^d} \Psi_0(x,v)f_0(x,v) dx dv \cr
& \quad = \int_{\R^d \times \R^d} \Psi(x,v,T)f^{\eta,\e}(x,v,T)dx dv + \int_0^T\int_{\R^d \times \R^d} f^{\eta,\e}(x,v,t)\partial_t\Psi(x,v,t)dxdvdt \cr
& \qquad -\int_0^T\int_{\R^d \times \R^d} (\nabla_x \Psi) \cdot v f^{\eta,\e} dx dvdt - \int_0^T\int_{\R^d \times \R^d} (\nabla_v \Psi) \cdot F^{\eta,\e}(f^{\eta,\e})f^{\eta,\e} dx dv dt.
\end{aligned}
\end{align}
We can easily show that the first, second, and third terms in the rhs of \eqref{weak-var} converge to their corresponding limits with $f$ instead of $f^{\eta,\e}$ due to their linearity and
$f^{\eta,\e} \to f$ in $\mathcal{C}([0,T];\pp_1(\R^d \times \R^d))$.  In order to show
\bq\label{est_conv1}
\int_0^T\int_{\R^d \times \R^d} (\nabla_v \Psi) \cdot F^{\eta,\e}(f^{\eta,\e})f^{\eta,\e} dx dv dt \to \int_0^T\int_{\R^d \times \R^d} (\nabla_v \Psi) \cdot F(f)f dx dv dt,
\eq
we provide two results.
\begin{lemma}\label{lem_conv11} $F(f)(x,v,t)$ is locally Lipschitz continuous in $(x,v)\in \R^d\times \R^d$ for all $t\geq 0$.
\end{lemma}

Its proof was given in subsection \ref{sec_2.5}. The next property follows by similar arguments.

\begin{lemma}\label{lem_conv2} The vector field 
\[
G(f^{\eta,\e})(y,w) := \int_{\R^d \times \R^d} \mb_{K(v)}(x-y) \lt( \nabla_v \Psi(x,v)\cdot (w-v)\rt) f^{\eta,\e}(x,v)\,dxdv
\]
is Lipschitz continuous in $(x,v)\in \R^d\times \tilde \Omega_0$ for all $t\geq 0$. Here the Lipschitz constant is independent of the regularization parameters $\eta$ and $\e$. 
\end{lemma}
\begin{proof}For $(y,w), (y',w') \in \R^d \times \tilde\om_0$, we obtain 
$$\begin{aligned}
&|G(f^{\eta,\e})(y,w) - G(f^{\eta,\e})(y',w')| \cr
&\quad = \lt|\int_{\R^d \times \R^d} \nabla_v \Psi(x,v)\cdot \lt(\mb_{K(v)}(x-y)(w-v) - \mb_{K(v)}(x-y')(w'-v)\rt)f^{\eta,\e}(x,v)\,dxdv\rt|\cr
&\quad \leq \lt|\int_{\R^d \times \R^d} \mb_{K(v)}(x-y) \lt(\nabla_v \Psi(x,v)\cdot \lt(w - w'\rt)\rt) f^{\eta,\e}(x,v)\,dxdv\rt|\cr
&\qquad + \lt|\int_{\R^d \times \R^d} \lt(\mb_{K(v)}(x-y) - \mb_{K(v)}(x-y')\rt) \lt(\nabla_v \Psi(x,v)\cdot (w'-v)\rt) f^{\eta,\e}(x,v)\,dxdv\rt|\cr
&\quad =: I + J,\cr
\end{aligned}$$
where $I$ is easily estimated by.
\[
I \leq \|\nabla_v \Psi\|_{L^\infty}|w - w'|.
\]
For the estimate of $J$, we first consider the case $|y - y'| \leq 1/2$. In this case, we obtain that
$$\begin{aligned}
J &\leq C\|\nabla_v \Psi\|_{L^\infty}\int_{\R^d \times \tilde\om_0} \mb_{\pa^{2|y -y'|}K(v)}(x-y)f^{\eta,\e}(x,v)\,dxdv \leq C\|\nabla_v \Psi\|_{L^\infty}\|f^{\eta,\e}\|_{L^\infty}|y-y'|,
\end{aligned}$$
 due to the compact support in velocity, ${\bf (H2)}$ and Lemma \ref{lem_est2} .
On the other hand, if $|y - y'| \geq 1/2$ we get
\[
J \leq C\|\nabla_v\Psi\|_{L^\infty} \leq C\|\nabla_v \Psi\|_{L^\infty} |y - y'|.
\]
This yields that for any $y,y' \in \R^d$ we have
\[
J \leq C\|\nabla_v \Psi\|_{L^\infty}(1 + \|f^{\eta,\e}\|_{L^\infty})|y-y'|,
\]
concluding the proof.
\end{proof}

We are now in a position to estimate the convergence \eqref{est_conv1}. Note that
$$\begin{aligned}
&\lt| \int_0^T \int_{\R^d \times \R^d} \nabla_v \Psi \cdot (F^{\eta,\e} (f^{\eta,\e})f^{\eta,\e} - F(f)f)\,dxdvdt\rt|\cr
& \quad = \lt| \int_0^T \int_{\R^d \times \R^d} \nabla_v \Psi \cdot (F^{\eta,\e} (f^{\eta,\e})f^{\eta,\e} - F(f^{\eta,\e})f^{\eta,\e})\,dxdvdt\rt| \cr
&\qquad + \lt| \int_0^T \int_{\R^d \times \R^d} \nabla_v \Psi \cdot (F (f^{\eta,\e})f^{\eta,\e} - F(f)f^{\eta,\e})\,dxdvdt\rt|\cr
&\qquad + \lt| \int_0^T \int_{\R^d \times \R^d} \nabla_v \Psi \cdot (F (f)f^{\eta,\e} - F(f)f)\,dxdvdt\rt|\cr
&\quad =:\sum_{i=1}^3 K_i.
\end{aligned}$$
$\diamond$ Estimate of $K_1$: A straightforward computation yields that
$$\begin{aligned}
&K_1 \leq \lt| \int_0^T \int_{\R^{2d} \times \R^{2d}} \nabla_v \Psi \cdot \lt(\mb_{K(v)}^{\eta,\e}(x-y) - \mb_{K(v)}^\eta(x-y) \rt)(w-v)f^{\eta,\e}(y,w)f^{\eta,\e}(x,v)\,dydwdxdvdt \rt|\cr
&\qquad+\lt| \int_0^T \int_{\R^{2d} \times \R^{2d}} \nabla_v \Psi \cdot \lt(\mb_{K(v)}^{\eta}(x-y) - \mb_{K(v)}(x-y) \rt)(w-v)f^{\eta,\e}(y,w)f^{\eta,\e}(x,v)\,dydwdxdvdt \rt|\cr
&\quad\leq C\|\nabla_v \Psi\|_{L^\infty}\int_0^T\int_{\R^{2d} \times \R^{2d}} \mb_{\pa^{2\e}K(v)}^\eta(x-y) f^{\eta,\e}(y,w)f^{\eta,\e}(x,v)\,dydwdxdvdt \cr
&\qquad +CT\|\nabla_v \Psi\|_{L^\infty}\|f^{\eta,\e}\|_{L^\infty}\eta\cr
&\quad \leq CT\|\nabla_v \Psi\|_{L^\infty}\|f^{\eta,\e}\|_{L^\infty}(\e + \eta).
\end{aligned}$$
Here, we have used the same arguments as in previous Lemma and in \eqref{est_ii} using the compact support in velocity, \eqref{lem_diff1}, \eqref{lem_add} and  ${\bf (H2)}$.
Thus we have $K_1 \to 0$ as $\e \to 0$ and $\eta \to 0$.

$\diamond$ Estimate of $K_2$: We first notice that
$$\begin{aligned}
K_2 &= \lt|\int_0^T \int_{\R^{2d} \times \R^{2d}} \nabla_v \Psi \cdot \mb_{K(v)}(x-y)(w-v)(f^{\eta,\e}(y,w) - f(y,w))f^{\eta,\e}(x,v)\,dxdydvdw dt\rt|\cr
&= \lt|\int_0^T \int_{\R^d \times \R^d} G(f^{\eta,\e})(y,w)(f^{\eta,\e}(y,w) - f(y,w))dydwdt\rt|,
\end{aligned}$$
and $\|G(f^{\eta,\e})\|_{W^{1,\infty}(\R^d \times \tilde\om_0)} \leq C\|\nabla_v \Psi\|_{L^\infty}$. Then it follows from duality in \eqref{blipd} that
\[
K_2 \leq Cd_1(f^{\eta,\e}, f) \to 0 \quad \mbox{as} \quad \eta,\e \to 0.
\]

$\diamond$ Estimate of $K_3$: Similarly, we deduce from Lemma \ref{lem_conv11} that
\[
K_3 = \lt|\int_0^T \int_{\R^d \times \R^d} \nabla_v \Psi \cdot F(f)(x,v) \lt(f^{\eta,\e}(x,v) - f(x,v)\rt)\,dxdvdt \rt| \leq Cd_1(f^{\eta,\e}, f) \to 0,
\]
as $\eta,\e \to 0$. Hence we conclude that $f$ is a solution of the system \eqref{sys_kin}. 

Uniqueness of weak solutions $f^{\eta,\e}$ just follows from the second statement in Proposition \ref{prop-grow}. To be more precise, let $f_1, f_2 \in L^{\infty}(0,T;(L_+^1\cap L^\infty)(\R^d \times \R^d)) \cap \mathcal{C}([0,T],\pp_1(\R^d \times \R^d))$ be the weak solutions to the system \eqref{sys_kin} with same initial data $f^0 \in (L_+^1 \cap L^\infty)(\R^d \times \R^d) \cap \pp_1(\R^d \times \R^d)$. Then Proposition \ref{prop-grow} yields that
\begin{equation*}
\frac{d}{dt}d_1(f_1(t),f_2(t)) \leq C(1 + \|f_1\|_{L^\infty_{x,v}})d_1(f_1(t),f_2(t)), \quad \mbox{for} \quad t \in [0,T].
\end{equation*}
This completes the proof of Theorem \ref{thm_weak}. Let us obtain further properties of the solutions.

\begin{corollary} Let $f$ be a weak solutions to \eqref{sys_kin} on the time-interval $[0,T)$. Then $f$ is determined as the push-forward of the initial density through the flow map generated by $(v,F(f))$.
\end{corollary}
\begin{proof}
Consider the following flow map:
\begin{equation}\label{flow}
\left\{ \begin{array}{ll}
\displaystyle \frac{d}{dt}X(t;s,x,v) = V(t;s,x,v), & \\[2mm]
\displaystyle \frac{d}{dt}V(t;s,x,v) = F(f)(X(t;s,x,v),V(t;s,x,v),t), &\\[4mm]
(X(s;s,x,v),V(s;s,x,v)) = (x,v),&
\end{array} \right.
\end{equation}
for all $s,t\in[0,T]$. Then since $f$ has compact support in $v$, the flow map \eqref{flow} is well-defined using the local Lispchitzianity of the velocity fields as in the proof of Proposition \ref{prop-grow}. Standard duality arguments using suitable test functions in the weak formulation of the solution leads to
\[
\int_{\R^d \times \R^d} h(x,v)f(x,v,t) dx dv = \int_{\R^d \times \R^d}h(X(0;t,x,v),V(0;t,x,v))f_0(x,v)dxdv,
\]
for $t \in [0,T]$ for all test functions $h$. This yields that $f$ is determined as the push-forward of the initial density through the flow map \eqref{flow}. 
\end{proof}

\begin{corollary}\label{cor_sta} Suppose that $f$ is a solution to the system \eqref{sys_kin} satisfying $\|f\|_{L^\infty(0,T;L^1 \cap L^\infty)} < \infty$ for some $T > 0$. Then for any global measure solution $\mu$ to the same system \eqref{sys_kin} with finite first order moment in $(x,v)$ we have
\[
d_1(f(t),\mu(t)) \leq d_1(f_0,\mu_0)\exp\lt(C\lt(1+\|f\|_{L^\infty(0,T;L^1 \cap L^\infty)}\rt)t\rt),
\]
for all $t \in [0,T]$.
\end{corollary}
\begin{proof} The proof can be easily obtained from the observation made in the second part of the statement in Proposition \ref{prop-grow} since we only require to use $L^1 \cap L^\infty$-norm of one of solutions to the system \eqref{sys_kin}.
\end{proof}
\begin{remark} Our strategies can also be applied to various other communicate weights. For example, we can consider the linear combination of different communication weights:
\[
\tilde \mb^m_{K(v)}(x-y):= \sum_{j=1}^m w_j \mb_{K^j(v)}(x-y) \qquad \mbox{for some} \quad m \geq 1, 
\]
where $w_j \in \R_+$ and $K^j(v)$ satisfies the assumptions {\bf (H1)}-{\bf (H2)}. \end{remark}


\section{Examples of Sensitivity Regions}\label{sec:5}
\subsection{A ball in $\R^d$}\label{sec:5_1}We consider the case $K(v) = B(0,r) := \{x \in \R^d : |x| \leq r\}$ with $r > 0$. 

In that case, we choose $\Theta(v) = \pa B(0,r)$. We notice that it is enough to check the condition $(\textbf{H2})$-$(ii)$ for the existence of solutions and the mean-field limit since the influence set $K$ is independent of the velocity variable. We can easily check that for $0 <\e \leq 1$
\begin{displaymath}
|\pa^\e B(0,r)| =\left\{ \begin{array}{ll}
\alpha(d)(r+\e)^d \leq \alpha(d)2^d \e^d \leq C(d)\e & \textrm{if $r \leq \e$,}\\[2mm]
\alpha(d)\lt( (r + \e)^d - (r-\e)^d\rt) = \alpha(d) 2\e \sum_{k=0}^{d-1}(r+\e)^k (r- \e)^{d-1-k} \leq C(d,r)\e  & \textrm{if $\e < r$},
  \end{array} \right.
\end{displaymath}
where $\alpha(d)$ is the volume of unit ball in $\R^d$. 
\subsection{A ball with radius evolving regularly with respect to velocity in $\R^d$} Let $\tilde{r}:\R_{+}\rightarrow \R_{+}$ be bounded and Lipschitz function, and consider the case $K(v)=B(0,\tilde{r}(|v|))$. 

In this case, it is clear to show that $B(0,\tilde{r}(|v|))$ satisfies ${\bf (H1)}$ due to the boundedness of function $\tilde r$. Moreover, by choosing $\Theta(v) = \pa B(0,\tilde{r}(|v|))$, we can easily verify that $B(0,\tilde{r}(|v|))$ satisfies ${\bf (H2)}$-$(i)$ and $(ii)$(see also previous section). Concerning the conditions ${\bf (H2)}$-$(iii)$ and $(iv)$, we notice that the symmetric difference $B(0,\tilde{r}(|v|)) \Delta B(0,\tilde{r}(|w|))$ has a form of torus which can also be expressed by the enlargement of $\pa B(0,\tilde{r}(|v|))$. Thus it is very clear that $B(0,\tilde{r}(|v|))$ satisfies ${\bf (H2)}$-$(iii)$ and $(iv)$ with the constant $C = \|\tilde r\|_{Lip}$.

\subsection{A vision cone in $\R^d$ with $d=2,3$} We consider $K(v)=C(r,v,\theta(|v|))$ which is given by
\[
C(r,v,\theta(|v|)) :=\lt\{ x : |x| \leq r \quad \mbox{and} \quad -\theta(|v|) \leq \cos^{-1}\lt(\frac{x \cdot v}{|x||v|} \rt) \leq \theta(|v|) \rt\},
\]
with $0 < \theta(z) \in \mc^\infty(\R_+)$ satisfying $\theta(z) = \pi$ for $0 \leq  z \leq 1$, $\theta(z)$ is decreasing for $z \geq 1$, and $\theta(z) \to \theta_* > 0$ as $|z| \to + \infty$.
\begin{figure}[lt]
        \centering
        \mbox{
         \subfigure[Vision cone]
         {\includegraphics[width=4.9cm,height=5.2cm]{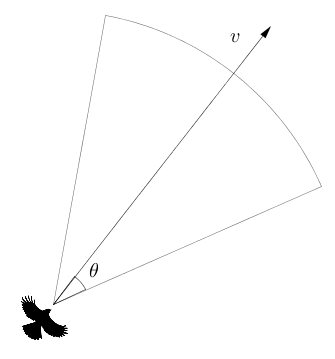}}
         \qquad
         \subfigure[Function of angle]
         {\includegraphics[width=8.0cm,height=5.0cm]{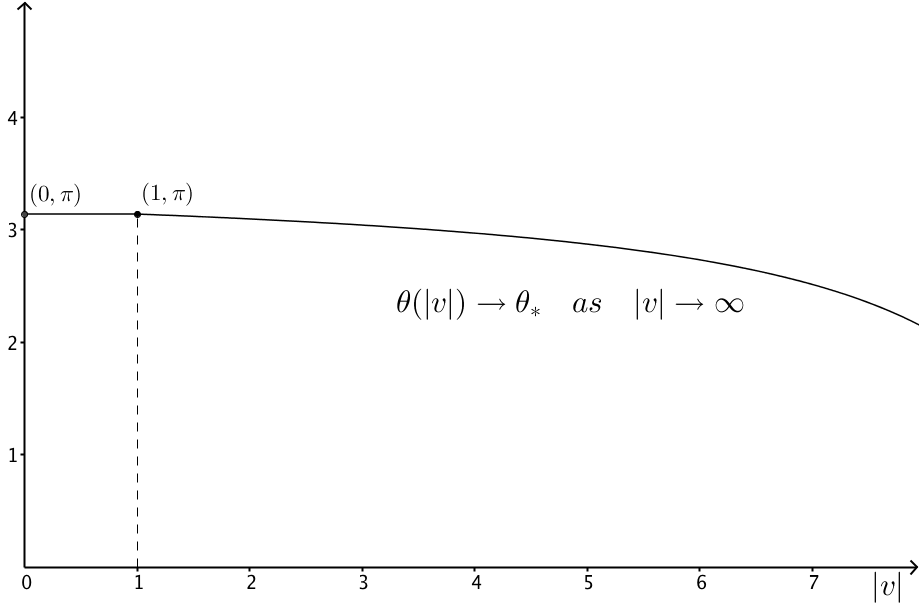}}
         }
         \caption{The vision cone is getting larger in covered angle as the speed gets smaller being a full ball at some chosen speed, for instance $|v|=1$.}
\end{figure}

In the remaining part of this subsection, we focus on checking the conditions $({\bf H1}) - ({\bf H2})$ for the case of vision cone.\newline

$\bullet$ $({\bf H1})$: It is obvious to get that $C(r,v,\theta(|v|)) \subseteq B(0,r)$ for all $v \in \R^d$ and $B(0,r)$ is compact set. \newline

We now define the family set $\Theta$ as
$$
\Theta(v):=\begin{cases}
\partial C(r,v,\theta(|v|))\cup R(v)& \text{ if } |v|\in(1/2,1), \\ 
\partial C(r,v,\theta(|v|)) & \text{ else }, 
\end{cases} 
$$    
where $R(v)=[a(v),b(v)]$ with
$$
a(v)=-r\frac{v}{|v|} \quad , \quad b(v)= 2r(|v|-1)\frac{v}{|v|}. 
$$
In other words, $R(v)$ is a segment with varies linearly from $\lt\{- r \frac v{|v|} \rt\}$  when $|v|=1/2$, to the whole segment $\bigl[0, - r \frac v{|v|} \bigr]  $ when $|v| = 1$. The family of sets $\Theta(v)$ interpolates continuously in some sense between the boundary of the ball for $|v|=1/2$ to the limit of the boundary of the cone as $\theta(|v|) \searrow \pi$ for $|v|\searrow 1$ that includes the interval $R(v)=[-rv,0]$, see middle picture in Figure \ref{fig2}.

\medskip
$\bullet$ $({\bf H2})$-$(i)$: Due to definition of $\Theta$, $\pa C(r,v,\theta(v)) \subset \Theta(v)$ for all $v \in \R^d$.
		
\medskip
$\bullet$ $({\bf H2})$-$(ii)$: This is also satisfied since $\Theta(v)$ is made of surfaces and lines, whose total area is bounded uniformly in $v$.  We give more detailed  estimates below.

Remark first that for all $|v|\leq 1$ and $\e\in(0,1)$
$$
|R(v)^{\e,+}| \le C \lt( \e|R(v)|+\e^d \rt) \leq C \e.
$$
Since for any sets $A,B\subset \R^d$,  $(A\cup B)^{\e,+}\subset A^{\e,+}\cup B^{\e,+}$, in order to check $({\bf H2})$-$(ii)$, it is enough to prove that 
$$
|(\partial C(r,v,\theta(|v|)))^{\e,+}|\leq C\e.
$$

We consider the following two cases:\newline

\begin{figure}[lt]
        \centering
        \mbox{
         {\includegraphics[width=8.5cm,height=8.5cm]{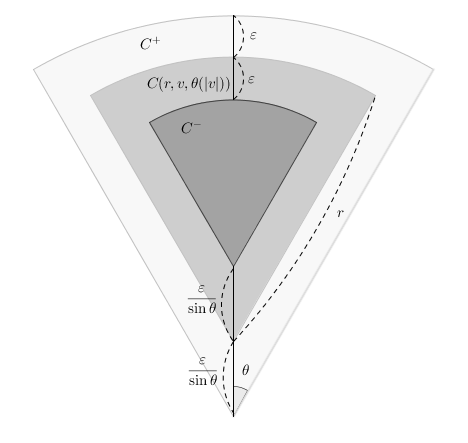}}}
         \caption{Case (i) \,\,$\theta(|v|) \leq \frac{\pi}{2}$.}
         \label{fig_4}
\end{figure}

$\diamond$ Case (i) $0 \leq \theta(|v|) < \pi/2$: By definition of $\theta$, we have that $0 < \theta_* \leq \theta(|v|) < \pi/2$. In two dimensions, we consider the two cones $C^- := C(r-\e - \e/\sin(\theta(|v|)), v, \theta(|v|))$ and $C^+ := C(r+\e + \e/\sin(\theta(|v|)), v, \theta(|v|))$ to measure the $\e$-boundary of $C(r,v,\theta(|v|))$ as
$$\begin{aligned}
|\pa^\e C(r,v,\theta(|v|))| &\leq |C^+| - |C^-| = \theta(|v|)\lt(r + \e + \frac{\e}{\sin\theta(|v|)} \rt)^2 - \theta(|v|)\lt(r - \e - \frac{\e}{\sin\theta(|v|)} \rt)^2\cr
&=4r\e\theta(|v|)\lt(1 + \frac{1}{\sin\theta(|v|)} \rt) \leq 2r\pi\e\lt(1 + \frac{1}{\sin\theta_*} \rt),
\end{aligned}$$
due to $C^{\e,+}(r,v,\theta(|v|)) \subseteq C^+$ and $C^- \subseteq C^{\e,-}(r,v,\theta(|v|))$, see Figure \ref{fig_4}. 

In a similar way, for the three-dimensional case, we obtain
$$\begin{aligned}
|\pa^\e C(r,v,\theta(|v|))| 
&\leq \frac{2\pi}{3}\lt( 1 - \cos\theta(|v|)\rt)\lt[\lt(r + \e + \frac{\e}{\sin\theta(|v|)} \rt)^3 - \lt(r - \e - \frac{\e}{\sin\theta(|v|)} \rt)^3\rt]\cr
&=\frac{4\pi}{3}\e\lt( 1 - \cos\theta(|v|)\rt)\lt(1 + \frac{1}{\sin\theta(|v|)}\rt)\lt( 3r^2 + \lt(1 + \frac{1}{\sin\theta(|v|)}\rt)^2\rt)\cr
&\leq \frac{4\pi}{3}\e\lt(1 + \frac{1}{\sin\theta_*}\rt)\lt( 3r^2 + \lt(1 + \frac{1}{\sin\theta_*}\rt)^2\rt).
\end{aligned}$$
Therefore we deduce that 
\[
\sup_{\lt\{v \in \R^d \,:\, \theta(|v|)\leq \pi/2\rt\}}|\pa^\e C(r,v,\theta(|v|))|\leq C\e
\]

\begin{figure}
\centering
\mbox{
         {\includegraphics[width=11.5cm,height=8.5cm]{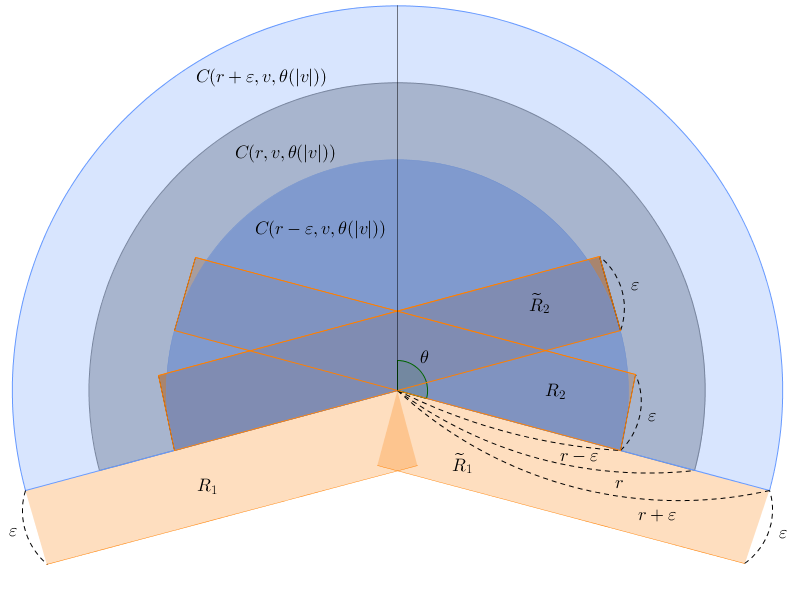}}}
         \caption{Case (ii) \,\,$\theta(|v|) \geq \frac{\pi}{2}$.}
         \label{fig_5}
\end{figure}

$\diamond$ Case (ii) $\pi \geq \theta(|v|) \geq \pi/2$: Using the same argument as above, in two dimensions, we find 
\[
C(r,v,\theta(|v|))^{\e,+} \subseteq C(r+\e,v,\theta(|v|)) \cup R_1 \cup \tilde R_1,
\]
and
\[
C(r - \e,v,\theta(|v|))\setminus(R_2 \cup \tilde R_2) \subseteq C(r,v,\theta(|v|))^{\e,-},
\]
where $R_i$ and $\tilde R_i$, $i=1,2$ are given in Figure \ref{fig_5}. Then this yields 
$$\begin{aligned}
|\pa^\e C(r,v,\theta(|v|))| &\leq |C(r+\e,v,\theta(|v|)) \cup R_1 \cup \tilde R_1| - |C(r - \e,v,\theta(|v|))\setminus(R_2 \cup \tilde R_2)|\cr
&\leq \theta(|v|)(r + \e)^2 + 2\e(r+\e) - \theta(|v|)(r - \e)^2 + 4\e r = 4r\e\theta(|v|) + 2\e(\e + 3r) \leq C\e.
\end{aligned}$$

Let us now consider $d=3$ and denote by $S(R_i)$ for $i=1,2$ the volume of the set obtained by rotating the rectangle $R_i$ about the $v$-vector. We set $R_0$ a rectangle with length $r+\e$ and width $\e$. Then we notice that $R_1$ can be obtained by rotating $R_0$ with respect to one of edges of $R_0$. If we consider that $R_0$ is placed in the right of $yz$ plane, then it is clear that the volume of the solid obtained by rotating the $R_0$ about the $y$-axis is given by $2\pi \int_{R_0} y\,dy\,dz$.
Then it follows from the definition of $R_0$ that there exists a $\tilde\theta \geq 0$ such that 
\[
S(R_1) \leq C\int_{R_0} |y \cos \tilde\theta + z \sin \tilde\theta| \,dy\,dz \leq \e\int_{R_0} |y|\,dy + \e |R_0| \leq C\e,
\]
for some positive constant $C >0$. Since this argument can be applied to other rectangles $ \tilde R_1, R_2, \tilde R_2$ , we get
\[
S(R_1) + S(R_2) + S(\tilde R_1) + S(\tilde R_2) \leq C\e.
\]

Then, in a similar fashion as the above, we find 
$$\begin{aligned}
|\pa^\e C(r,v,\theta(|v|))| &\leq C\lt( (r+\e)^3 - (r-\e)^3 \rt) + C\e \leq 2C\e(3r^2 + \e^2) + C\e \leq C\e.
\end{aligned}$$
This concludes 
\[
\sup_{\lt\{v \in \R^d \,:\, \theta(|v|)\geq \pi/2\rt\}}|\pa^\e C(r,v,\theta(|v|))| \leq C\e.
\]
Combining the estimates in Case (i) and (ii), we conclude that 
\[
\sup_{v \in \R^d}|\pa^\e C(r,v,\theta(|v|))| \leq C\e.
\]

$\bullet$ $({\bf H2})$-$(iii)$ and $(iv)$ :
	The vision cone is an interior domain, delimited by the union of some smooth surfaces. In that case, the symmetric difference between two different cone $K(v) \Delta K(w)$ is the set in-between the two boundaries $\pa K(v)$ and $\pa K(w)$. Since $\pa K(w) \subset \Theta(w)$ by ${\bf (H2)}$-$(i)$ if $\pa K(v) \subset \Theta(w)^{\ep,+}$, we also have $K(v) \Delta K(w) \subset  \Theta(w)^{\ep,+}$, or equivalently in this particular case 
$({\bf H2})$-$(iii)$ is a consequence of $({\bf H2})$-$(iv)$. Thus we focus on the hypothesis $({\bf H2})$-$(iv)$ in the rest of the proof. For this we consider three steps as follows. 

\medskip
\paragraph{ \bf Step 1. The case $|v| = |w|$}
We first check the properties in the simpler case where $|v|=|w|$. It allows to define the rotation $R_{v,w}$ with the only rotation (with an axis orthogonal to $v$ and $w$ if $d=3$) that maps $v$ to $w$. In view of the definition of $\Theta$, it is clear that $\Theta(w) = R_{v,w} \Theta(v)$. 

Now, inside the ball $B(0,r)$, the maximal displacement $d_{max}$ made when we use $R_{v,w}$ is 
\[
d_{max} := \sup_{x \in B_r} |R_{v,w}x -x | = \lt| \frac r{|v|} v  - \frac r{|v|}w \rt| =  \frac r{|v|} |  v  - w |.
\]
This implies that $\Theta(v) \subset \Theta(w)^{d_{max},+}$. When $|v| = |w| \ge 1/2$, $d_{max} \le 2 r |v-w|$ and we get 
\[
\Theta(v) \subset \Theta(w)^{2r |v-w|,+}.
\]
This is also true when $|v| = |w| \le 1/2$ because $\Theta(v)=\Theta(w)=\partial B(0,r)$ in that case. In fact, it is especially for that reason that we need a vision cone equals to the full ball for small velocities. 

\medskip
\paragraph{ \bf Step 2. The case $v=\lambda w$ with $\lambda > 0$}

When both $|v|, |w| \le 1$, it is clear to get $R(v) \subset R(w)^{2r|v-w|,+}$ 
since the function $b(v)$ is 2r-Lipschitz, and subsequently this implies $\Theta(v) \subset \Theta(w)^{2r|v-w|,+}$.

When $|v|,|w|\geq 1$, the set $\Theta(v)$ and $\Theta(w)$ are boundaries of cones with different angles, and we get
\[
\Theta(v) \subset \Theta(w)^{r | \theta(v) - \theta(w)|,+} 
\subset \Theta(w)^{r \| \theta\|_{Lip} | v - w|,+}.
\]

In the remaining cases, we can combine the two above case introducing $w' = \frac v {|v|} = \frac w {|w|}$. Since the three points $v,w',w$ are aligned in that order, $|v-w| = |v-w'| + |w'-w|$, the following inclusions hold with $L=(2 \vee \|\theta\|_{Lip}) r $:
\[
\Theta(v) \subset \Theta(w')^{2r|v-w'|,+ } \subset 
\Theta(w)^{2r|v-w'|+ \|\theta\|_{Lip} r|w'-w |,+} 
\subset \Theta(w)^{L|v-w|,+}.
\]
In fact, the above inclusions are valid in the case $|v| \le 1 \le |w|$, but the final inclusion is valid in both case.

\medskip
\paragraph{ \bf Step 3. The general case.}
For any $v,w \in \R^d$, we introduce the middle point $\tilde v$ such that
\[
|\tilde v| = |w| \quad \mbox{and} \quad \tilde v = \lambda v \quad \mbox{with} \quad \lambda \in (0,1),
\]
when $|v| \ge |w|$ and the point defined in the same way but exchanging the role of $v$ and $w$ when $|w| \ge |v|$.
This middle point satisfies 
\[ | v - \tilde v| \le |v-w|, \quad  \text{and} \quad | \tilde v - w| \le |v- w|.
\]
Using the two steps above, we conclude that
\[\Theta(v) \subset \Theta(\tilde v)^{L|v-\tilde v|,+ }
 \subset \Theta(w)^{ L |v-\tilde v| +  2r|\tilde v-w|,+ }
\subset \Theta(w)^{2 L |v-w|,+ },
\] 
which concludes the verification of $({\bf H2})$-$(iv)$.

%
%

\section{Further extensions}
In this section, we present generalizations of main results with suitable modifications for much more general models: 
\begin{equation}\label{cs}
\left\{ \begin{array}{ll}
\pa_t f + v \cdot \nabla_x f + \nabla_v \cdot (F(f)f) = 0, \qquad (x,v)\in \R^d \times \R^d, \quad t > 0,& \\[2mm]
f(x,v,t)|_{t=0}=f_0(x,v)\qquad (x,v) \in \R^d \times \R^d,&  
\end{array} \right.
\end{equation}
where the force term $F(f)$ can be chosen from two different types:
\begin{equation}\label{force}
F(f)(x,v,t) =
\left\{ \begin{array}{ll}
\displaystyle \int_{\R^d \times \R^d} \psi(x-y)\mb_{K(v)}(x-y) h(w - v) f(y,w)\,dydw
& \mbox{(Cucker-Smale type)}\\[4mm] 
\displaystyle \int_{\R^d \times \R^d} \nabla_x \varphi(x-y)\mb_{K(v)}(x-y) f(y,w)\,dydw & \mbox{(Attractive-Repulsive type)}.  
\end{array} \right.
\end{equation}
Here $\psi$, $h$, and $\varphi$ denote the communication weight, velocity coupling, and interaction potential, respectively. We refer to \cite{HT,CFTV,AIR,MT,MT2}.

Similarly as before, we introduce the particle approximation of the kinetic equation \eqref{cs}:
\begin{equation}\label{cs_ode}
\left\{ \begin{array}{ll}
\dot{X}_i(t) = V_i(t), \quad i=1,\cdots,N,~~t > 0, &  \\[2mm]
\displaystyle \dot{V}_i(t) = \left\{ \begin{array}{ll}
\displaystyle \sum_{j \neq i} m_j \psi(X_i-X_j)\mb_{K(V_i)}(X_i-X_j) h(V_j - V_i) 
& \mbox{(Cucker-Smale type)},\\[4mm] 
\displaystyle \sum_{j \neq i} m_j  \nabla_x \varphi(X_i-X_j)\mb_{K(V_i)}(X_i-X_j) & \mbox{(Attractive-Repulsive type)},
\end{array} \right.
&\\[8mm]
\left( X_i(0), V_i(0) \right) =: \left( X_i^0, V_i^0 \right),\quad i=1,\cdots,N. & 
\end{array} \right.
\end{equation}
 
Then we define the differential inclusion system with respect to the above ODE system \eqref{cs_ode} and the empirical measure $\mu^N$ associated to a solution to this differential inclusion system in a similar fashion with \eqref{sys_di} and \eqref{def_emp}.

By a similar strategy as in Theorems \ref{thm_weak} and \ref{main}, we have the existence of weak solutions to the kinetic equations and mean-field limit.
\begin{theorem}\label{thm_gene} Given an initial data compactly supported in velocity satisfying \eqref{main_ass} and assume that the sensitivity region set-valued function $K(v)$ satisfies ${\bf (H1)}$-${\bf (H2)}$. Then there exists a positive time $T > 0$ such that the system \eqref{cs}-\eqref{force} with $\psi, h, \nabla_x\varphi \in W^{1,\infty}(\R^d)$ admits a unique weak solution $f$ in the sense of Definition {\rm\ref{def_weak}} on the time interval $[0,T]$, which is also compactly supported in velocity. Moreover, $f$ is determined as the push-forward of the initial density through the flow map generated by the Lipschitz velocity field $(v, F(f))$ in phase space and the solutions satisfies the stability estimate \eqref{main_stab}.

Furthermore, we have the estimate of mean-field limit such that 
\[
d_1(f(t),\mu^N(t))\le e^{Ct} d_1(f(0),\mu^N(0)) \quad \mbox{for all} \quad t \in [0,T],
\]
there exists a constant $C$ depending only on $T$, $f_0$, and $d$.
\end{theorem}
\begin{proof}Since the proof is very similar to that of Theorems \ref{thm_weak} and \ref{main} as mentioned before, we only give essential parts of the proof. Let us also consider the Cucker-Smale type force in \eqref{force} for $F(f)$ being the attractive-repulsive force treated analogously.

$\bullet$ {\it Support estimate of the density in velocity:} For the existence of weak solutions to the kinetic equation \eqref{cs}-\eqref{force}, we regularize the solution similarly as in \eqref{reg-k-CS-si} and define the forward $Z^{\eta,\e}(s) := \left( X^{\eta,\e}(s;0,x,v), V^{\eta,\e}(s;0,x,v)\right)$ with regularization parameters $\eta$ and $\e$ satisfying the following ODE system:
$$
\begin{aligned}
\frac{dX^{\eta,\e}(s)}{ds} &= V^{\eta,\e}(s),\cr
\frac{dV^{\eta,\e}(s)}{ds} &= F^{\eta,\e}(X^{\eta,\e},V^{\eta,\e},s) = \int_{\R^d \times \R^d} \psi(x-y)\mb_{K(V^{\eta,\e})}^{\eta,\e}(X^{\eta,\e}-y) h(w - V^{\eta,\e}) f^{\eta,\e}(y,w)\,dydw.
\end{aligned}
$$
Note that the force term $F^{\eta,\e}$ is bounded from above by $\|(\psi,h)\|_{L^\infty}\|f_0\|_{L^1}$ and this yields
\[
R^{\eta,\e}_v(t) \leq R^0_v + \|(\psi,h)\|_{L^\infty}\|f_0\|_{L^1} t,
\]
i.e., the support of $f^{\eta,\e}$ in velocity linearly increases in time $t$, however it does not depend on the regularization parameters $\eta$ and $\e$. 

$\bullet$ {\it Lipschitz continuity of the force field $F(f)$}: Let us assume that there exists a solution $f$ to the equation \eqref{cs}-\eqref{force} in the sense of Definition \eqref{def_weak} and $f$ has compact support in velocity. Then we obtain
$$\begin{aligned}
&F(f)(x,v) - F(f)(\tilde x, \tilde v)\cr
&= \int_{\R^d \times \R^d} \Big(\psi(x-y)\mb_{K(v)}(x-y)h(w - v) - \psi(\tilde x-y)\mb_{K(\tilde v)}(\tilde x-y)h(w - \tilde v)\Big)f(y,w)\,dydw\cr
&= \int_{\R^d \times \R^d} \lt( \psi(x-y) - \psi(\tilde x - y)\rt)\mb_{K(v)}(x-y)h(w-v) f(y,w)\,dydw\cr
&\quad + \int_{\R^d \times \R^d} \psi(\tilde x-y)\lt(\mb_{K(v)}(x-y) - \mb_{K(\tilde v)}(x-y)\rt)h(w-v)f(y,w)\,dydw\cr
&\quad + \int_{\R^d \times \R^d}  \psi(\tilde x-y)\lt( \mb_{K(\tilde v)}(x-y) - \mb_{K( \tilde v)}(\tilde x-y)\rt)h(w-v)f(y,w)\,dydw\cr
&\quad + \int_{\R^d \times \R^d} \psi(\tilde x-y) \mb_{K(\tilde v)}(\tilde x - y)\lt(h(w - v) - h(w- \tilde v)\rt)f(y,w)\,dydw\cr
&=: \sum_{i=1}^4 I_i.
\end{aligned}$$

In comparsion to \eqref{est_lip}, we just need to estimate an additional term. By using the similar argument as in \eqref{est_ii} together with the regularity of communication weight $\psi$ and velocity coupling $h$, we estimate each term $I_i,i=1,\cdots,4$ as
$$\begin{aligned}
I_1 &\leq \|(\nabla_x \psi, h)\|_{L^\infty}\|f_0\|_{L^1}|x - \tilde x|, \cr
I_2 &\leq C(\|\rho\|_{L^\infty}+\|f_0\|_{L^1})(\psi, h)\|_{L^\infty}|v - \tilde v|,\cr
I_3 &\leq C(\|\rho\|_{L^\infty}+\|f_0\|_{L^1})(\psi,h)\|_{L^\infty}| x - \tilde x|, \cr
I_4 &\leq \|(\psi,\nabla_x h)\|_{L^\infty}\|f_0\|_{L^1}|v - \tilde v|,\cr
\end{aligned}$$
where $(\bf{H1})$-$(\bf{H2})$ have been used. Hence we have 
\[
|F(f)(x,v) - F(f)(\tilde x, \tilde v)| \leq C|(x,v) - (\tilde x, \tilde v)|,
\]
and this concludes the proof.
\end{proof}

\begin{remark}Our strategy is also applicable to first-order swarming models with sensing zones \cite{CPT}. More precisely, let $\rho = \rho(x,t)$ be the probability density of individuals at position $x \in \R^d$ and time $t > 0$ satisfying
\begin{equation}\label{f_conti}
\left\{ \begin{array}{ll}
\pa_t \rho + \nabla_x \cdot (\rho u) = 0, &x \in \R^d, \quad t > 0,\\[2mm]
\displaystyle u(x,t) := \int_{\R^d} \mb_{K(w(x))} (x-y) \nabla_x \varphi(x-y)\rho(y)\,dy, &x \in \R^d, \quad t > 0, \\[4mm]
\rho(x,t)|_{t=0}=\rho_0(x), & x \in \R^d,
\end{array} \right.
\end{equation}
where $\varphi \in W^{1,\infty}(\R^d)$ and $w$ is a given orientational field satisfying $w \in W^{1,\infty}(\R^d)$ and $|w| \geq w_0 > 0$. Here the velocity field $u$ is non-locally computed in terms of the density $\rho$. Suppose the sensing zone $K(w(x))$ satisfies our assumptions $(\bf{H1})$-$(\bf{H2})$, then the continuity equation \eqref{f_conti} is well approximated by 
\[
\left\{ \begin{array}{ll}
\displaystyle \dot{X}_i(t) = \sum_{j \neq i}m_j \nabla_x\varphi(X_i - X_j)\mb_{K(w(X_i))}(X_i - X_j), &\\[5mm]
X_i(0)=X_i^0, \qquad i=1,\cdots,N.& 
\end{array} \right.
\]
via the associated differential inclusion system in the sense of Theorem \ref{thm_gene}. Note that taking $K$ to be $C(r,w(x),\theta_0)$ with $\theta_0 \in (0,\pi]$ satisfies these conditions. 
\end{remark}

%
%
%
%
\section*{Acknowledgments}
\small{JAC was partially supported by the project MTM2011-27739-C04-02
DGI (Spain) and from the Royal Society by a Wolfson Research Merit Award. YPC was partially supported by Basic Science Research Program through the National Research Foundation of Korea (NRF) funded by the Ministry of Education, Science and Technology (2012R1A6A3A03039496). YPC also acknowledges the support of the ERC-Starting Grant HDSPCONTR ``High-Dimensional Sparse Optimal Control''. JAC and YPC were supported by EPSRC grant with reference
EP/K008404/1.}


%
%
%

\end{document}